\documentclass[12pt]{article}
\usepackage{amscd}
\usepackage{latexsym,amsmath,amssymb,amsbsy, amsthm}
\usepackage[square]{natbib}

\usepackage{subcaption}

\usepackage{hyperref}
\usepackage{lscape,fancyhdr,fancybox}
\usepackage{graphicx,epsfig, xy}
\usepackage{color}
\usepackage[hmarginratio=1:2, vmarginratio =5:5,
textheight=22cm,bindingoffset=1.5cm, textwidth=14.6cm]{geometry}
\usepackage{tikz}

\newcommand{\comment}[1]{}

\numberwithin{equation}{section}

\newtheorem{remark}{Remark}[section]
\newtheorem{theorem}{Theorem}[section]

\newtheorem{lemma}{Lemma}[section]

\theoremstyle{definition}
\newtheorem{definition}{Definition}[section]

\newtheorem{result}{Result}[section]

\usepackage{amsfonts}

\usepackage{txfonts}
\usepackage{graphics}

\DeclareMathOperator{\Tr}{Tr}

\renewcommand{\1}{\bf 1}

\newcommand{\beq}{\begin{eqnarray}}
\newcommand{\eeq}{\end{eqnarray}}
\newcommand{\ben}{\begin{eqnarray*}}
\newcommand{\een}{\end{eqnarray*}}

\allowdisplaybreaks

\begin{document}
 \def\shortauthors{A. Bose, K. Saha, P. Sen}
\title{\textbf{\Large \sc
\Large{Some patterned matrices with independent entries} 
}}\small

 \author{
 \parbox[t]{0.20
\textwidth}{{\sc Arup Bose}
 \thanks{Statistics and Mathematics Unit, Indian Statistical Institute, 203 B.T. Road, Kolkata 700108, India. email: bosearu@gmail.com. 
Research  supported by J.C. Bose National Fellowship, Department of Science and Technology, Govt. of India.}}
\parbox[t]{0.25\textwidth}{{\sc Koushik Saha}
\thanks{Department of Mathematics, Indian Institute of Technology Bombay, Mumbai,  India. email: koushik.saha@iitb.ac.in.
Research  supported by MATRICS Grant of Science \& Engineering Research Board, Department of Science and Technology, Govt. of India.}}
 \parbox[t]{0.25\textwidth}{{\sc Priyanka Sen}
 \thanks{Statistics and Mathematics  Unit, Indian Statistical Institute, 203 B.T. Road, Kolkata 700108, INDIA. email: priyankasen0702@gmail.com}}
}

\date{\today}  
\maketitle
\begin{abstract}
Patterned random matrices such as the reverse circulant, the symmetric circulant, the Toeplitz and the Hankel matrices
and their almost sure limiting spectral distribution (LSD), have attracted much attention. Under the assumption that the entries are taken from an  i.i.d. sequence with finite variance, the LSD are tied together by a common thread{\textemdash}the $2k$th moment of the limit equals a weighted sum over different types of pair-partitions of the set $\{1, 2, \ldots, 2k\}$ and are universal. Some results are also known for the sparse case. 

In this paper we  generalise these results by relaxing significantly the i.i.d. assumption. For our models, the limits are defined via a larger class of partitions and are also not universal. Several existing and new results for patterned matrices, their band and sparse versions, as well as for matrices with continuous and discrete variance profile follow  as special cases.

\end{abstract}
\medskip

\section{Introduction}\label{introduction} Suppose $A_n$ is an $n\times n$ real symmetric random matrix with (real) eigenvalues $\lambda_1,\lambda_2,\ldots,\lambda_n$. Its \textit{empirical spectral mesure} is the random probability measure:
\begin{align*}
\mu_{A_n} = \frac{1}{n} \sum_{i=1}^{n}  \delta_{\lambda_i},
\end{align*}
where $\delta_x$ is the Dirac measure at $x$. The probability distribution function, $F^{A_n}$, known as the \textit{empirical spectral distribution} (ESD) of $A_n$ is given by 
$$F^{A_n}(x,y)= \frac{1}{n}\displaystyle \sum_{i=1}^n \boldsymbol{1} (\Re(\lambda_i)\leq x, \Im(\lambda_i)\leq y).$$
Clearly the ESD is a random distribution on $\mathbb{C}$. However, if all the eigenvalues of $A_n$ are real, then the ESD is a function on $\mathbb{R}$ and is given by
$$F^{A_n}(x)= \frac{1}{n}\displaystyle \sum_{i=1}^n \boldsymbol{1} (\lambda_i\leq x).$$
The expected empirical spectral distribution (EESD), denoted by $\mathbb{E}[F^{A_n}]$ is also a distribution function and is non-random. The probability measure corresponding to the EESD will be denoted by $\mathbb{E}\mu_{A_n}$.

 The notions of convergence, as  $n\to\infty$, for the ESD used in this article are, the \textit{weak convergence of the EESD} and the \textit{weak convergence of the ESD} (either in \textit{probability} or \textit{almost surely}). The limits are identical when the latter limits are non-random. In any case, any of these limits will be referred to as the \textit{limiting spectral distribution} (LSD) of $\{A_n\}$. 

We shall consider the reverse circulant matrix, the symmetric circulant matrix, the symmetric Toeplitz matrix and the  Hankel matrix  where the entries of the matrices are independent for each fixed $n$ but not identically distributed. 

\vspace{.2cm}
\noindent{\bf Reverse circulant matrix:} An $n\times n$  reverse circulant matrix is defined as 
$$
RC_n=\left(\begin{array}{ccccc}
x_0 & x_1 & x_2 & \cdots   & x_{n-1} \\
x_1 & x_2 & x_3 & \cdots   & x_{0}\\
x_2 & x_3 & x_4 & \cdots   & x_{1} \\
\vdots & \vdots & {\vdots} & \ddots & \vdots \\
x_{n-1} & x_0 & x_1 & \cdots   & x_{n-2}
\end{array}\right).
$$
For $1\leq j \leq n-1$, its $(j+1)$-th row is obtained by giving its $j$-th row a left circular shift by one position. The matrix is  symmetric and the $(i,\;j)$-th element of the matrix is $x_{(i+j-2) \bmod  n}$.

\vspace{.2cm}
\noindent{\bf Symmetric circulant matrix:} An $n\times n$  symmetric circulant matrix is defined by
$$
SC_n=\left(\begin{array}{cccccc}
x_0 & x_1 & x_2 & \cdots  & x_1 \\
x_1 & x_0 & x_1 & \cdots  & x_{2}\\
x_{2} & x_1 & x_0 & \cdots & x_{3}\\
\vdots & \vdots & {\vdots} & \ddots  & \vdots \\
x_1 & x_2 & x_3 & \cdots  & x_0
\end{array}\right).
$$
The $(i,\;j)$-th element of the matrix is $x_{\frac{n}{2}-|\frac{n}{2}-|i-j||}$. Observe that for $j=1,2,\ldots, n-1$, its $(j+1)$-th row is obtained by giving its $j$-th row a right circular shift by one position. 

\vspace{.2cm}
\noindent{\bf Symmetric Toeplitz matrix:} The symmetric Toeplitz matrix is defined as 
$$
T_n=\left(\begin{array}{cccccc}
x_0 & x_1 & x_2 & \cdots &  x_{n-1} \\
x_1 & x_0 & x_1 & \cdots &  x_{n-2}\\
x_2 & x_1 & x_0 & \cdots &  x_{n-3} \\
\vdots & \vdots & {\vdots} & \ddots  & \vdots \\
x_{n-1} & x_{n-2} & x_{n-3} & \cdots  & x_{0}
\end{array}\right).
$$
The $(i,\;j)$-th element of the matrix is $x_{|i-j| }$. Note that the symmetric circulant matrix is a Toeplitz matrix with the added restriction that $x_{n-j}=x_j$.

\vspace{.2cm}
\noindent{\bf Hankel matrix :} An $n\times n$  Hankel matrix is defined as 
$$
H_n=\left(\begin{array}{cccccc}
x_2 & x_3 & x_4 & \cdots &  x_{n+1} \\
x_3 & x_4 & x_5 & \cdots &  x_{n+2}\\
x_4 & x_5 & x_6 & \cdots &  x_{n+3} \\
\vdots & \vdots & {\vdots} & \ddots  & \vdots \\
x_{n+1} & x_{n+2} & x_{n+3} & \cdots  & x_{2n}
\end{array}\right).
$$
The $(i,\;j)$-th element of the matrix is $x_{(i+j) }$.

If the variables $\{x_i\}$ are i.i.d. with mean $0$ and variance $1$, then for all these matrices, after scaling by $n^{-1/2}$, the ESD converges weakly almost surely as $n\to \infty$ and in each case the LSD does not depend on the underlying distributions. More details about the LSD are given later.

A natural question is what happens if the entries do not come from an i.i.d. sequence  and/or the distribution of the entries varies with $n$. One such question is addressed in \citep{zakharevich2006generalization} where the author considered the Wigner matrix with entries that are i.i.d for every $n$ with certain moment conditions and concluded the convergence of the ESD in  probability.  We also address this question for the above mentioned matrices in this article. Keeping only the independence assumption, under appropriate moment conditions on the entries of the matrices, we prove several LSD results for these matrices. In the process we show that unlike the i.i.d. case where only certain pair partitions are involved in the LSD, many other interesting partition classes are involved in the independent case. Another fall out is that the LSD are not fully universal but are dictated by the moment conditions assumed. The LSD results for the i.i.d. case alluded to in the previous paragraph can be deduced from our results. 

Here is the plan of the paper. In Section \ref{existing result} we give a description of the existing results for some patterned matrices with i.i.d input. In Section \ref{mainresults} we state our main results on the LSD of these matrices when the input is independent but not neceesarily i.i.d.  In Section \ref{Circuits and Words} we briefly recall the 
well-known ideas of link function, words/partitions etc. that are crucial in applying the moment method. In Section \ref{contributing words} we identify the classes of partitions which contribute to the limit moments for the different matrices. Section \ref{proofs of the theorems} contains the proofs of the theorems. Section \ref{example and discussion} provides some insight into the limits as well as application to   matrices with discrete and continuous variance profiles, band matrices, triangular matrices and sparse matrices. 

\section{Existing results}\label{existing result}
We first recall the existing results on the LSD of these patterned matrices.
\citep{bose2002limiting} and then \citep{bose2008another} showed the convergence of the ESD of $n^{-1/2}RC_n$ to the LSD (described in Result \ref{result:rev}),  weakly in probability and  almost surely respectively.
We refer the readers to  \citep{bose2018patterned} for a detailed proof.

\begin{result}\label{result:rev} Suppose that the entries $\{x_n;n\geq 0\}$ are  i.i.d. with mean 0 and  variance 1. Then, as $n \to \infty$, the almost sure LSD of $\frac{1}{\sqrt{n}}RC_n$ is the symmetrised Rayleigh distribution. This law, say $\mathcal{R}$,  has the following density
$$f(x)= |x| \exp(-x^2), \ \ x \in \mathbb{R}.$$
\end{result}

The moments $\beta_{k}(\mathcal{R})$ of $\mathcal{R}$ are given by 
$$\beta_k(\mathcal{R})= \left\{
\begin{array}{cc}
0 &\mbox{if}\ k \ \text{is odd},\\
k! & \mbox{if}\ k \ \mbox{is even}.
\end{array}
\right.$$
The LSD of $n^{-1/2}SC_n$  was first studied in \citep{bose2008another}. 
See  Theorem 2.4.2 in \citep{bose2018patterned} for a detailed proof and a short history on the precursors of the following result.
\begin{result}\label{result:sym} Suppose that the entries $\{x_{n}; n\geq 0\}$ are  i.i.d. with mean 0 and variance 1.  Then, as $n \to \infty$, the almost sure LSD of $\frac{1}{\sqrt n}SC_n$ is the standard normal distribution.
\end{result}

The study of the random Toeplitz  and Hankel matrices were initiated in a seminal paper by \citep{bai1999}. \citep{hammond2005distribution} and \citep{bryc2006spectral} established the LSD of $n^{-1/2}T_n$.  The LSD of $n^{-1/2}H_n$ was  established by \citep{bryc2006spectral} and  \citep{liu2011limit} with different techniques.
We refer to  \citep{bose2018patterned} for a detailed proof of the following result on LSD of Toeplitz and Hankel matrices.
\begin{result}\label{result:toe} Suppose that the entries $\{x_{n}; n\geq 0\}$ are i.i.d. with mean 0 and variance 1. Then, 
as $n \to \infty$, the almost sure LSD of $\frac{1}{\sqrt n}T_n$ and $\frac{1}{\sqrt n}H_n$ exist, say $\mathcal{L}_T$ and $\mathcal{L}_H$ respectively, and they are symmetric about 0.
\end{result}

\section{Main results}\label{mainresults}
Before we state our main results we need the notion of \textit{multiplicative extension} of a sequence of numbers on the set of partitions of $\{1,2,\ldots, k\}$.
Let $[k]$ denote the set $\{1,2,\ldots, k\}$ and  $\mathcal{P}(k)$ denote the set of all partitions of $[k]$. 
Suppose $\{c_n\}_{n\geq 1}$ is any sequence of real numbers. Its multiplicative extension is defined on the set of all partitions  $\mathcal{P}(k)$ of $[k]$, $k \geq 1$  
as follows. For any partition $\sigma\in \mathcal P(k)$, define 
$$c_{\sigma}=\prod c_{|V|},$$
where the product is taken over all blocks $V$ of the partition $\sigma$ and $|V|$ denotes the cardinality of the set $V$. 

We define two types of partitions of $[k]$  
which will appear in the LSD of symmetric circulant and reverse circulant matrices, respectively.

\vspace{0.2cm} 
\noindent\textbf{Even partition:} A partition $\sigma=\{V_1,V_2,\ldots,V_j\}$ of $[k]$ is said to be an even partition if each $V_i$ has  even number of integers. The set of all even partitions of $[k]$ is denoted by $E(k)$ and the  set of all even partitions of $[k]$ with $b$  blocks is denoted by $E_b(k)$. 

For example, $\{\{1,3\},\{2,4,5,6\}, \{7,10\},\{8,9\}\}\in E_4(10)$ is an even partition of $[10]$. Observe that if $k$ is odd, then $[k]$ does not have an even partition.

\vspace{0.2cm}
\noindent\textbf{Symmetric partition:}  A partition $\sigma=\{V_1,V_2,\ldots,V_j\}$ of $[k]$ is said to a symmetric partition if each $V_i$ has same number of odd and  even  integers. The set of all symmetric partitions of $[k]$ is denoted by $S(k)$ and  the set of all symmetric partitions of $[k]$ with $b$ blocks is denoted by $S_b(k)$.

For example, $\{\{1,4\},\{2,3,5,6\},\{7,10\},\{8,9\}\}\in S_4(10)$ is a symmetric partition of $[10]$, but  $\{\{1,3\},\{2,4,5,6\}, \{7,10\},\{8,9\}\}\notin S(10)$. Also observe that $S(k)\subset E(k)$.
\vspace{0.2cm}
 
Now suppose that the distribution of the entries for the $n$th matrix depends on $n${\textemdash}they come from a triangular array of sequences $\{x_{i,n};0\leq i\leq n \ \text{or}\ 2n\}_{n\geq 1}$.  To keep the notation simple, we shall often write $x_i$ for  $x_{i,n}$. We introduce the following set of assumptions on the entries.

\vspace{0.2cm}
\noindent \textbf{Assumption A.} 
Let $\{g_{k,n};0\leq k\leq n\}$ be a sequence of bounded Riemann integrable functions on $[0,1]$. Suppose there exists a sequence $\{t_n\}$ with $t_n\in [0,\infty]$ such that 
\begin{enumerate}
\item [(i)] for each $k \in \mathbb{N}$,
\begin{align}
 & n \ \mathbb{E}\left[x_{i}^{2k}\boldsymbol {1}_{\{|x_{i}|\leq t_n\}}\right]= g_{2k,n}\Big(\frac{i}{n}\Big) 
 \ \ \ \text{for } \ 0\leq i \leq n-1, \label{gkeven}\\
& \displaystyle \lim_{n \rightarrow \infty} \ n^{\alpha} \underset{0\leq i  \leq n-1}{\sup} \ \mathbb{E}\left[x_{i}^{2k-1}\boldsymbol {1}_{\{|x_{i}|\leq t_n\}}\right] = 0  \ \ \text{for any } \alpha<1 . \label{gkodd}
\end{align}
\item [(ii)] The functions $g_{2k,n}$ converge uniformly to functions $g_{2k}$ for all $k \geq 1$. 
\item[(iii)] Let  $M_{2k}=\|g_{2k}\|$ (where $\|\cdot\|$ denotes the sup norm) and $M_{2k-1}=0$ for all $k \geq 1$. Suppose $\alpha_{2k}=  \sum_{\sigma \in \mathcal{P}(2k)} M_{\sigma}$  satisfy \textit{Carleman's condition}, 
$$ \displaystyle \sum_{k=1}^{\infty} \alpha_{2k}^{-\frac{1}{2k}}= \infty.$$
\end{enumerate}

Now we state our first theorem.
\begin{theorem}\label{thm:mainrev}
Consider $RC_n$ whose entries $\{x_i;0\leq i< n\}$ are independent and  satisfy Assumption A. Let $Z_n$ be the reverse circulant matrix with the entries $y_{i}= x_{i}\boldsymbol {1}_{\{|x_{i}|\leq t_n\}}$. Then
\begin{enumerate}
\item[(a)] the EESD of $Z_n$ converges weakly to a symmetric probability measure $\mu$, say. The moment sequence of $\mu$ is given by
$$\beta_k(\mu) =\left\{ \begin{array}{cc}
\displaystyle\sum_{\sigma \in S(k)} C_{\sigma} & \text{if } k \text{ is even},\\
0 &  \text{if } k \text{ is odd},
\end{array}
\right.
$$ 
where $C_{2m}= \int_{0}^1 g_{2m}(t) \ dt, m \geq 1$ and $S(k)$ denotes the set of all symmetric partitions of $[k]$.

In particular, if for every $n$, $\{x_{i,n};1\leq i \leq n\}$ are i.i.d., then the above holds with $C_{2k}=\lim g_{2k,n}$ (which are now constant functions).
\item[(b)] Further if 
\begin{align}\label{truncation}
 \sum_{i=0}^{n-1} \mathbb{E}[ x_{i}^2\boldsymbol {1}_{\{|x_{i}| > t_n\}}] \rightarrow 0, 
\end{align}
then the EESD of $RC_n$ converges weakly to $\mu$. 
 \end{enumerate} 
\end{theorem}
 Next we deal with the symmetric circulant matrix with independent entries. 
\begin{theorem}\label{thm:mainsym}
Consider  $SC_n$  whose entries $\{x_{i} ;0 \leq i < n\}$ are independent and  satisfy Assumption A. Let $Z_n$ be the  $n\times n$ symmetric circulant matrix with entries $y_{i}= x_{i}\boldsymbol {1}_{\{|x_{i}|\leq t_n\}}$. Then
\begin{enumerate}
\item[(a)] the EESD of $Z_n$ converges weakly to a symmetric probability measure $\mu$, say and  the moments of $\mu$ are given by
$$\beta_k(\mu) =\left\{ \begin{array}{cc}
\displaystyle\sum_{\sigma \in E(k)} a_{\sigma}C_{\sigma} & \text{ if } k \text{ is even},\\
0 &  \text{if } k \text{ is odd},
\end{array}
\right.
$$ where $C_{2m}= 2 \int_{0}^{\frac{1}{2}} g_{2m}(t) \ dt, m \geq 1$ are constants determined by the functions $\{g_{2k}, \ k \geq 1\}$, $a_{2n}=\frac12{2n \choose n}$ and $E(k)$ is the set of all even partitions of $[k]$. 

In particular, if for every $n$, $\{x_{i,n}\}$, $1\leq i \leq n$ are i.i.d., then the above holds with $C_{2k}=\lim g_{2k,n}$.
\item[(b)] Further if 
\begin{align}\label{truncationsym}
 \sum_{i} \ \mathbb{E}[x_{i}^2\boldsymbol {1}_{\{|x_{i}| > t_n\}}] \rightarrow 0, 
\end{align}
then the EESD of $SC_n$ converges weakly to $\mu$.
 \end{enumerate} 
\end{theorem}
The following  theorem is for the Topelitz and Hankel matrices with independent entries.
\begin{theorem}\label{thm:maintoe}
 Consider $T_n$ (respectively, $H_n$) whose entries $\{x_i;0\leq i<n\}$ are independent and  satisfy Assumption A. 
Let $Z_n$ be the  $n\times n$ Toeplitz matrix (respectively, Hankel matrix)  with entries $y_i=x_i\boldsymbol {1}_{\{|x_{i}|\leq t_n\}}$. Then
\begin{enumerate}
\item[(a)] the EESD of $Z_n$ converges weakly to a symmetric probability measure $\mu^{\prime}$ (respectively, $\tilde{\mu}$) say,  whose moment sequence is determined by the functions $g_{2k}, \ k \geq 1$.
\item[(b)] Further if 
\begin{align}\label{truncationtoe}
 \sum_{i} \ \mathbb{E}[x_{i}^2\boldsymbol {1}_{\{|x_{i}| > t_n\}}] \rightarrow 0, 
\end{align}
then the ESD of $T_n$ (respectively, $H_n$) converges weakly to $\mu^{\prime}$ (respectively, $\tilde{\mu}$) almost surely (or in probability).
 \end{enumerate} 
\end{theorem}

In Section \ref{example and discussion}, we shall have a discussion on our main assumption (Assumption A)  and our results. In particular, we shall show how Results \ref{result:rev}{\textemdash}\ref{result:toe}  can be derived from our results and new LSD results arise for different variants of these matrices. We shall also see how results on band matrices, triangular matrices and matrices with variance profiles follow from the arguments used in the proofs of the above theorems.

\section{Preliminaries}\label{Circuits and Words}
 Let 
\begin{align*}
L_n:\{1,2,\ldots ,n\}^2 \longrightarrow \mathbb{Z}, \ \,n \geq 1
\end{align*}
be a sequence of functions, called \textit{link functions}. Often we write $L_n=L$ for convenience. All four matrices can be represented using link functions as follows: 
\begin{itemize}
\item Reverse circulant matrix: $RC_n=\big(x_{L(i,j)}\big)$ where $L(i,j)=(i+j-2)\bmod n$.
\item Symmetric circulant matrix: $SC_n=\big(x_{L(i,j)}\big)$ where $L(i,j)=\frac{n}{2}- \big|\frac{n}{2}-|i-j|\big|$.
\item Symmetric Toeplitz matrix: $T_n=\big(x_{L(i,j)}\big)$ where $L(i,j)=|i-j|$.
\item Hankel matrix: $H_n=\big(x_{L(i,j)}\big)$ where $L(i,j)=i+j$.  
\end{itemize}
Now we define a property of the link functions  which will play an important role.

\vspace{0.2cm}
Define $$\Delta(L)= \displaystyle \sup_n \sup_{t \in \mathbb{Z}} \sup_{1 \leq k \leq n} \#\{l: 1\leq l \leq n, L(k,l)=t\} < \infty.$$
Observe that $\Delta(L)=1$ for reverse circulant and Hankel matrices, and $\Delta(L)=2$ for symmetric circulant and Toeplitz matrices. 

\vspace{.2cm}  
\noindent A \textbf{circuit} $\pi$ is a function $\pi : \{0,1,2,\ldots,k\}\longrightarrow \{1,2,3,\ldots,n\}$ with $\pi(0)=\pi(k)$. We say that the \textit{length} of  $\pi$ is $k$ and denote it by $\ell(\pi)$. Suppose $A_n$ is any of the four patterned matrices with link function $L$, that is, $A_n=\big(x_{L(i,j)}\big)$. Then using  circuits, we can express the trace, $\Tr(A_n^k)$ as
\begin{align}\label{momentf1usual}
\Tr(A_{n}^{k})  = \sum_{\pi:\ell(\pi)=k}x_{L(\pi(0),\pi(1))}x_{L(\pi(1),\pi(2))}\cdots x_{L(\pi(k-1),\pi(k))}=\sum_{\pi:\ell(\pi)=k}X_{\pi},
\end{align}
where $X_{\pi}= x_{L(\pi(0),\pi(1))}x_{L(\pi(1),\pi(2))}\cdots x_{L(\pi(k-1),\pi(k))}$. For any $\pi$, the values $L(\pi(i-1),\pi(i))$ will be called \textit{edges}. When an edge appears more than once in a circuit $\pi$, then it is called 
\textit{matched}. Any $m$ circuits $\pi_1,\pi_2,\ldots,\pi_m$ are said to be \textit{jointly-matched} if each edge occurs at least twice across all circuits. They are said to be \textit{cross-matched} if each circuit has an edge which occurs in at least one of the other circuits. \\

\noindent \textbf{Equivalence of Circuits}: Circuits $\pi_1$ and $\pi_2$ are \textit{equivalent} if and only if for all $1 \leq i,j \leq k$,
\begin{align*}
L(\pi_1(i-1),\pi_1(i))=L(\pi_1(j-1),\pi_1(j))
\Longleftrightarrow L(\pi_2(i-1),\pi_2(i))=L(\pi_2(j-1),\pi_2(j)).
\end{align*} 
The above is an equivalence relation on $\{\pi:\ell(\pi)=k\}$. Any equivalence class  of circuits can be indexed by an element of $\mathcal{P}(k)$. The positions where the edges match are identified by each block of a partition of $[k]$. For example, the partition $\{\{1,3\}, \{2,4,5\}\}$ of $[5]$ corresponds to the equivalent class 
$$\big\{\pi:\ell(\pi)=5\ \mbox{and}\ L(\pi(0),\pi(1))=L(\pi(2),\pi(3)),\ L(\pi(1),\pi(2))=L(\pi(3),\pi(4))=L(\pi(4),\pi(5))\big\}.$$  
Also, an element of $\mathcal{P}(k)$ can be identified with a \textit{word} of length $k$ of letters.  Given a partition, we represent the integers of the same partition block by the same letter, and   the first occurrence of each letter is in alphabetical order and vice versa. For example, the partition $\{\{1,3\}, \{2,4,5\}\}$ of $[5]$ corresponds to the word $ababb$. On the other hand, the word $aabccba$ represents the partition $\{\{1,2,7\},\{3,6\},\{4,5\}\}$  of $[7]$. A typical word will be denoted by $\boldsymbol{\omega}$ and its $i$-th letter as $\boldsymbol{\omega}[i]$. For example, for the word $\boldsymbol{\omega}=abbacac$, $\boldsymbol{\omega}[2]=b$ and the partition is $\{\{1,4,6\},\{2,3\},\{5,7\}\}$.\\

\noindent \textbf{The class $\Pi(\boldsymbol {\omega})$}: Attached to any word $\boldsymbol {\omega}$, is an equivalence class of circuits $\Pi(\boldsymbol {\omega})$: 
\begin{align*}
\Pi(\boldsymbol {\omega}) &= \Big\{\pi: \boldsymbol {\omega}[i]=\boldsymbol {\omega}[j] \Leftrightarrow L(\pi(i-1),\pi(i))= L(\pi(j-1),\pi(j))\Big\}.
\end{align*}
This implies that for any word $\boldsymbol{\omega}$ of length $k$, the cardinality of $\pi(\boldsymbol{\omega})$ is
\begin{align}\label{Pi(omega)}
\big|\Pi(\boldsymbol {\omega})\big| = \big|\big\{ &\big(\pi(0), \pi(1),\ldots,\pi(k)\big): 1\leq \pi(i)\leq n \text{ for } i=0,1,\ldots,k,\ \pi(0)=\pi(k), \nonumber\\
 &\quad L(\pi(i-1),\pi(i))= L(\pi(j-1),\pi(j)) \text{ if and only if }\boldsymbol {\omega}[i]=\boldsymbol {\omega}[j]  \big\}\big|.
\end{align}
The following two classes of words mirror the partition classes $E(k)$ and $S(k)$ of Section \ref{mainresults}.\\
  
\noindent\textbf{Even word:}
A word $\boldsymbol {\omega}$ is called an \textit{even} word if each distinct letter in $\boldsymbol {\omega}$ appears an even number of times. This is the same as saying that the corresponding partition is even. We shall continue to use $E(2k)$ to denote the set of all even words  of length $2k$. For example, $ababcc$ is an even word of length 6 with 3 letters. The corresponding partition of $[6]$ is $\{\{1,3\},\{2,4\},\{5,6\}\}$.\\
		
\noindent\textbf{Symmetric word:}
A word $\boldsymbol {\omega}$ is  \textit{symmetric} if each distinct letter appears equal number of times in odd and even positions. This is the same as saying that the corresponding  partition is symmetric. We shall continue to use $S(2k)$ to denote the set of all symmetric  words of length $2k$.
For example, $abbbba$ is a symmetric word of length $6$ with $2$ distinct letters. Note that a symmetric word is always even. Hence $S(2k) \subset E(2k) \subset \mathcal{P}(2k)$. But every even word is not symmetric. For example, $ababcc$ is even but not symmetric.\\

\noindent\textbf{Vertex and generating vertex:} Any $\pi(i)$ of a circuit $\pi$ will be called a \textit{vertex}. It is a \textit{generating vertex} if $i=0$ or $\boldsymbol {\omega}[i]$ is the first occurrence of a letter in the  word $\boldsymbol {\omega}$ corresponding to $\pi$. All other vertices are \textit{non-generating}. 

Note that the circuits corresponding to a word $\boldsymbol {\omega}$ are completely determined by the generating vertices: once we have determined the generating vertices, there are only finitely many choices for the non-generating vertices. In particular, for the matrices considered here, the non-generating vertices are linear combinations of the generating vertices due to the structure of the link functions. This can be shown by induction.

The first vertex $\pi(0)$ is always generating and after that there is one generating vertex for each new letter in $\boldsymbol{\omega}$. So, if $\boldsymbol {\omega}$ has $b$ distinct letters then the number of generating vertices is $(b+1)$.  From \eqref{Pi(omega)} we observe that the growth of $|\Pi(\boldsymbol {\omega})|$ is determined by the number of generating vertices that can be chosen freely. For some words, depending on the link function, some of these vertices may not be chosen freely, that is some of the generating vertices might be a linear combination of the other generating vertices. Also for all the four matrices, since $\Delta(L)\leq 2$, we immediately conclude that
\begin{equation}\label{cardiality of word}
| \Pi(\boldsymbol {\omega})| = \mathcal{O}( n^{b+1})\ \ \text{whenever} \ \omega \text{ has } \\   b \ \ \text{distinct letters}.
\end{equation}  
As we shall see later, the existence of 
\begin{equation}\label{existence of positive limit}
\lim_{n\to\infty}\frac{| \Pi(\boldsymbol {\omega})|}{n^{b+1}}\end{equation} 
is tied to the LSD. We explore the limit of \eqref{existence of positive limit} in the next section. 

\section{Existence of \ $\lim_{n\to\infty}\frac{| \Pi(\boldsymbol {\omega})|}{n^{b+1}}$}\label{contributing words}
    
\subsection{$\lim_{n\to\infty}\frac{| \Pi(\boldsymbol {\omega})|}{n^{b+1}}$ for Reverse Circulant} 
		
	\begin{lemma}\label{lem:rev} For each word $\boldsymbol {\omega}$ with $b$ distinct letters
				\begin{eqnarray} 
		 \lim_{n \rightarrow \infty}\frac{1}{n^{b+1}} | \Pi(\boldsymbol {\omega})|=  
\begin{cases} 1 \ \ \text{\rm if} \ \ \boldsymbol {\omega} \ \ \text{is symmetric},\\
0\ \ \text{\rm otherwise.}
\end{cases}
\end{eqnarray}

	 		\end{lemma}
			
		\begin{proof}
First note the following. Let 
\begin{equation*}
t_i= \pi(i)+ \pi(i-1)  \text{ for } 1 \leq i \leq 2k.
\end{equation*}
Clearly,  $\boldsymbol {\omega}[i]=\boldsymbol {\omega}[j]$ if and only if
$$(\pi(i-1)+\pi(i)-2)\bmod n =  (\pi(j-1) +  \pi(j)-2)\bmod n,$$ 
that is,  
		$$t_{i} -  t_{j}  =  0,\ n \text{ or }  -n.$$ 
Now fix an $\boldsymbol {\omega}$ with $b$ distinct letters. Suppose the distinct letters appear for the first time at the positions $i_1,i_2,\ldots,i_b$. Clearly $t_{i_1}=t_1$. If the first letter appears again in the $j$-th position, then
\begin{equation*}
t_j = t_1 \ (\bmod \  n). 
\end{equation*}
Similarly,  for every $i,\ 1 \leq i \leq 2k$,
\begin{equation}\label{rc1}
t_i\ =\ t_{i_j}\ (\text{mod}\ n) \  \  \ \text{for some }j \in \{1,2,\ldots,b\}.
\end{equation}
We break the proof into four steps.\vskip3pt

\noindent 
\textbf{Step 1}. If $\{\pi(i_j), 0 \leq j \leq b\}$ can be chosen freely (where $i_0=0$), then  $\{\pi(0),t_{i_j};  1 \leq j \leq b\}$ do not satisfy any non-trivial linear relation.

To prove this, assume the contrary. 
Then there exists constants $\alpha_j, 0\leq j \leq b$, which are not simultaneously zero, and a constant $c$ such that 
\begin{equation}\label{t_i-expressn}
\alpha_0\pi(0)+\alpha_1t_{i_1}+ \alpha_2 t_{i_2} + \cdots + \alpha_b t_{i_b} =c.
\end{equation}
Choose the greatest index $j$ (say $m$) such that $\alpha_m \neq 0$. If $m= 0$, then $\pi(0)=c/\alpha_0$ which is not possible as it has free choice. So $m\neq 0$. Now $t_{i_m}= \pi(i_m)+\pi(i_m-1)$ where  $\pi(i_m-1)$ is a linear combination of $\{\pi(i_s): 0 \leq s \leq m-1\}$.  Also the other $t_{i_j}$'s with non-zero $\alpha_j$ can be written as linear combinations of  $\{\pi(i_s): 0 \leq s \leq m-1\}$, using the formula for $t_{i_j}$ and the fact that all non-generating vertices are linear combination of the generating vertices. 

Then substituting these values in \eqref{t_i-expressn}, we get a linear equation in $\pi(i_j)$'s where the coefficient of $\pi(i_m)$ is $\alpha_m \neq 0$. Hence $\pi({i_j})$'s satisfy a non-trivial linear equation which is not possible, because $\{\pi(i_j), 0 \leq j \leq b\}$ can be chosen freely. 
Hence the claim is proved. 

\vskip3pt

\noindent 
\textbf{Step 2}.
Suppose $\pi(0)$ and $t_{i_j}, 1\leq j \leq b$ can be chosen freely. Then the generating vertices $\{\pi(i_j) : 0 \leq j \leq b\}$ do not satisfy any non-trivial linear relation.

Suppose $\{\pi(i_j) : 0 \leq j \leq b\}$ satisfy a non-trivial linear relation. 
So there exists $\alpha_0,\alpha_1,\alpha_2,\ldots,\alpha_b$, not all zero, and a constant $c$ such that 
\begin{align}\label{pi-expressn}
\alpha_0 \pi(i_0)+\alpha_1 \pi(i_1)+ \cdots +\alpha_b \pi(i_b)=c.
\end{align}
Choose the greatest index $j$, say $m$, such that $\alpha_m \neq 0$. Now $\pi({i_m})= t_{i_m}-\pi(i_m-1)$ where  $\pi(i_m-1)$ is a linear combination of $\{\pi(0), t_{i_j}: 1 \leq s \leq m-1\}$. Also other $\pi_{i_j}$'s can be written as a linear combination of $\{t_{i_s}: 1 \leq s \leq b\}$ and $\pi(0)$ using the formula for $t_{i_j}$. Using these in \eqref{pi-expressn}, $\pi(0)$ and $t_{i_j}$'s satisfy a non-trivial linear relation. 
But that cannot be possible as they can be chosen freely. 
  This proves the claim. 
\vskip3pt

\noindent \textbf{Step 3}.
Suppose $\pi(0)$ and $t_{i_j},1\leq j\leq b$  can be chosen freely. Then the word is symmetric.

To see this, if $\boldsymbol{\omega}$ is of length $2k$, then for any corresponding circuit 
\begin{equation}\label{rc2}
(t_1+t_3+\cdots+t_{2k-1}) - (t_2+t_4+\cdots+t_{2k})= \pi(0)- \pi(2k)= 0.
\end{equation}
Therefore, using \eqref{rc2} and \eqref{rc1}, we see that there exists $\alpha_j \in \mathbb{Z}$ for all $1 \leq j \leq b$ such that  
\begin{equation*}
\alpha_1t_{i_1}+ \alpha_2 t_{i_2} + \cdots + \alpha_b t_{i_b} = 0\ (\text{mod } n).
\end{equation*}
For the $t_{i_j}$'s to have free choice, we must have $\alpha_j=0$
 for all $j \in \{1,2,\ldots,b\}$. Thus for each $j$,
\begin{align*}
\big|\{l:l \text{ odd, }t_l=t_{i_j}\ (\text{mod } n)\}\big|=\big|\{l:l \text{ even, }t_l=t_{i_j}\ (\text{mod } n) \}\big|.
\end{align*} 
 So each letter appears equal number of times at odd and even places and the word is symmetric.

Now if the length of the word is odd, say $2k+1$, then $$(t_1+t_2+\cdots+t_{2k+1})-(t_2+t_4+\cdots+t_{2k})=\pi(2k+1)+\pi(0)=2\pi(0).$$  
Now substituting $t_i$ by $t_{i_j}$'s using \eqref{rc1} in the above equation, we see that $\pi(0),t_{i_1},\ldots,t_{i_b}$ satisfies a non-trivial linear relation. Therefore, $\pi(0)$ and $t_{i_j},1\leq j\leq b$ cannot be chosen freely. This contradicts our assumption. Hence the length of the word cannot be odd. 
 
 
 So, if the word is not symmetric, then using Steps 1 and 2 it follows that at least one of the generating vertices cannot be chosen freely and hence has at most a finite number of choices as $n\to \infty$. 
As a result, 
$$\lim_{n \rightarrow \infty}\frac{1}{n^{b+1}} | \Pi(\boldsymbol{\omega})|= 0\ \ \text{for any non-symmetric word} \ \boldsymbol{\omega}.$$
\noindent \textbf{Step 4}.
Suppose $\boldsymbol {\omega}$ is symmetric with $b$ distinct letters. Then $\pi({i_j})$ $(0 \leq j \leq b)$ can be chosen freely and then the non-generating vertices can be chosen uniquely. First we choose $\pi(0)$ and $\pi(1)$ freely. Next, if $\pi(2)$ is a generating vertex, we choose it freely else we consider the relation between $t_1$ and $t_2$. Clearly as $\pi(2)$ is not a generating vertex, $\boldsymbol {\omega}[2]$ is an old letter and hence $t_1=t_2$ which implies $\pi(0)=\pi(2)$. 

  To prove that the other non-generating vertices can be chosen uniquely, we argue inductively from left to right.
Suppose $\pi(i)$ is a non-generating vertex  and the non-generating vertices among $\pi(k)\ (1 
 \leq k \leq i-1)$ have already been chosen uniquely. Suppose  the $j$-th distinct letter appears at the $i$-th  position. Then we know from \eqref{rc1} that
 \begin{equation}\label{rc3}
 \pi(i)  =  \pi(i_j-1)+\pi(i_j)-\pi(i-1) \ (\text{mod}\ n),\ \ \ \ \text{for some } j \ \ \text{ where }i_j\leq i.
 \end{equation}
Therefore  $\pi(i) = A-n, A, A+n,$ where $A=\pi(i_j-1)+\pi(i_j)-\pi(i-1)$.
Observe that $-(n-1) \leq A \leq 2n$, and only one of the values $A-n,A,A+n$ can be between 0 and $n-1$. Hence $\pi(i)$ can be determined \textit{uniquely} from \eqref{rc3} as $\pi(i_j-1),\pi(i_j),\pi(i-1)$ have  already been determined.

As a consequence of the above we get, 
$$ \lim_{n \rightarrow \infty}\frac{1}{n^{b+1}} | \Pi(\boldsymbol{\omega})|= 1\ \ \text{for each symmetric word}\ \boldsymbol {\omega}.$$
%
\noindent This completes the proof of the lemma.
 \end{proof}
	
\subsection{$\lim_{n\to\infty}\frac{| \Pi(\boldsymbol {\omega})|}{n^{b+1}}$ for Symmetric Circulant}

\begin{lemma}\label{lem:sym} For each word $\boldsymbol {\omega}$ with $b$ distinct letters 
				\begin{eqnarray} 
		 \lim_{n \rightarrow \infty}\frac{1}{n^{b+1}} | \Pi(\boldsymbol {\omega})|=  
\begin{cases} \prod_{i=1}^b {{k_i-1} \choose \frac{k_i}{2}} \ \ \text{\rm if} \ \ \boldsymbol {\omega} \ \ \text{is even},\\
0\ \ \text{\rm otherwise.}
\end{cases}
\end{eqnarray}
where $k_i$ denotes the number of times the $i$-th distinct letter appeared in $\boldsymbol {\omega}$.
\end{lemma}

\begin{proof}
  
First note the following. Let $$ s_i= \pi(i)- \pi(i-1) \  \   \text{for }1 \leq i \leq 2k.$$
Clearly, $\boldsymbol {\omega}[i]=\boldsymbol {\omega}[j]$ 
if and only if $|n/2 - |s_i||  = \ |n/2 - |s_j||$, that is, 
	$ s_i - s_j = 0,\ n \ \text{or}\ -n,$ or $s_i  +  s_j =  0,\ n \  \text{or}\ -n.$
	

Now we fix an $\boldsymbol {\omega}$ with $b$ distinct letters which appear at $i_1,i_2,\ldots,i_b$ positions for the first time. 
Clearly $s_{i_1}=s_1$. If the first letter also appears in the $j$-th position, then
\begin{align*}
s_j =  s_1 \ (\text{mod}\ n) \mbox{ or } s_j = - s_1 \ (\text{mod}\ n) .
\end{align*}
Similarly, for every $i,\ 1 \leq i \leq 2k$,
\begin{align}
s_i =    s_{i_j}\ (\text{mod}\ n)\ \mbox{ or }\ s_i\ =  \ - s_{i_j}\ (\text{mod}\ n), \  \  \text{for some }j \in \{1,2,\ldots,b\}.\label{scminus}
\end{align} 
We break the proof into four steps as before.

%
\vskip3pt
\noindent 
\textbf{Step 1}. If $\{\pi(i_j), 0 \leq j \leq b\}$ can be chosen freely (where $i_0=0$), then  $\{\pi(0),s_{i_j};  1 \leq j \leq b\}$ 
do not satisfy any non-trivial linear relation.

To see this, assume the contrary. 
 Then there exists constants $\alpha_j, 0 \leq j \leq b$, which are not simultaneously zero and a constant $c$ such that 
	 \begin{equation}\label{si-expressn}
		 \alpha_0\pi(0)+\alpha_1s_{i_1}+ \alpha_2 s_{i_2} + \cdots + \alpha_b s_{i_b} = c.
\end{equation}
Choose the greatest index $j$, say $m$, such that $\alpha_m \neq 0$. If $m=0$, then $\pi(0)=c/ \alpha_0$ which is not possible as $\pi(0)$ has free choice. So $m \neq 0$. Now $s_{i_m}= \pi(i_m)-\pi(i_m-1)$ where  $\pi(i_m-1)$ is a linear combination of $\{\pi(i_s): 0 \leq s \leq m-1\}$. Also the other $s_{i_j}$'s with non-zero $\alpha_j$ can be written as linear combination of  $\{\pi(i_s):0 \leq s \leq m-1\}$ using the formula for $s_{i_j}$ and the fact that all non-generating vertices are linear combination of the generating vertices.

Then substituting these values in \eqref{si-expressn}, observe that we get a linear equation in $\pi(i_j)$'s where the coefficient of $\pi(i_m)$ is $\alpha_m \neq 0$. Hence $\pi({i_j})$'s satisfy a non-trivial linear equation but that is impossible as $\{\pi(i_j), 0 \leq j \leq b\}$ can be chosen freely. This proves the claim.

\vskip3pt

\noindent \textbf{Step 2}. Suppose $\pi(0)$ and $s_{i_j},1\leq j\leq b$ can be chosen freely. Then  the generating vertices $\{\pi(i_j), 0 \leq j \leq b\}$ do not satisfy any non-trivial linear relation.

 Assume the contrary. So there exists $\alpha_0,\alpha_1,\alpha_2,\ldots,\alpha_b$ not all zero and a constant $c$ such that 
\begin{align}\label{pi-expressn2}
\alpha_0 \pi(i_0)+\alpha_1 \pi(i_1)+ \cdots +\alpha_b \pi(i_b)=c
\end{align}
Choose the greatest index $j$ (say $m$) such that $\alpha_m \neq 0$. Now $\pi({i_m})= s_{i_m}+\pi(i_m-1)$ where  $\pi(i_m-1)$ is a linear combination of $\{\pi(0),s_{i_q}: 1 \leq q \leq m-1\}$. Also the other $\pi_{i_j}$'s can be written as a linear combination of $\{s_{i_t}: 1 \leq t \leq b\}$ and $\pi(0)$ using the formula for $s_{i_j}$. Using these relations in \eqref{pi-expressn2}, $s_{i_j}$'s satisfy a non-trivial linear relation. But that is impossible as they are chosen freely. 
 This proves the claim.
\vskip3pt

	
\noindent \textbf{Step 3}. Suppose $\pi(0)$ and $s_{i_j}, 1\leq j\leq b$ can be chosen freely. Then the word is even.

To see this, observe that if the length of the word is $k$ then, for any corresponding circuit $\pi$,
\begin{equation*}\label{sc2}
 \sum_{i=1}^{k} s_i = \pi(0) - \pi(k)= 0.
\end{equation*}
Therefore, using \eqref{sc2} and
\eqref{scminus}, we see that there exists $\alpha_j \in \mathbb{Z}$ for all $1 \leq j \leq b$ such that  
\begin{equation*}
\alpha_1s_{i_1}+ \alpha_2 s_{i_2} + \cdots + \alpha_b s_{i_b} =0 \ (\text{mod }n).
\end{equation*}
For $s_{i_j}$'s to have free choice, we must have $\alpha_j=0$ for all $j \in \{1,2,\ldots,b\}$.
Therefore for each $j$, 
 \begin{align}\label{even}
\big|\{l:s_l=s_{i_j} \ (\text{mod } n)\}\big|=\big|\{l:s_l=-s_{i_j}\ (\text{mod } n)\}\big|.
 \end{align} 
  That is, each letter appears an even number of times and the word is even.
So, if the word is not even, then using Steps 1 and 2 it follows that at least one of the generating vertices cannot be chosen freely and hence has at most a finite number of choices as $n \to \infty$. As a result,
$$\lim_{n \rightarrow \infty}\frac{1}{n^{b+1}} | \Pi(\boldsymbol {\omega})|=0 \ \ \text{if}\  \boldsymbol {\omega}  \ \text{is not even}.$$
\noindent \textbf{Step 4}. Suppose that $\boldsymbol {\omega}$ is an even word of length $2k$ with $b$ distinct letters. Then $\pi(i_j)\ ( 0 \leq j \leq b)$ can be chosen freely and then the non-generating vertices can be chosen uniquely. 

First we choose $\pi(0)$ and $\pi(1)$ freely. Next if $\pi(2)$ is a generating vertex, we choose it freely else we consider the relation between $s_1$ and $s_2$. Clearly as $\pi(2)$ is not a generating vertex, the letter in the second position of $\boldsymbol{\omega}$ is old and hence either $s_1=s_2$ or $s_2=-s_1$. If the former is true then $\pi(2)=2\pi(1)-\pi(0)\pm n$ else $\pi(2)=\pi(0)$. In any case we see that there is only one value of $\pi(2)$ between $1$ and $n$. 

Now we shall show that all other non-generating vertices can be chosen uniquely. We argue inductively from left to right. Suppose $\pi(i)$ is a non-generating vertex and the non-generating vertices among $\pi(l)\ (1 
 \leq l \leq i-1)$ has been chosen uniquely. Suppose the $j$-th distinct letter appears at the $i$-th position. Then we know from \eqref{scminus} that
 \begin{align}\label{sc3}
 \pi(i) & =  \pm(\pi(i_j-1)-\pi(i_j))+\pi(i-1) \ (\text{mod}\ n),\ \ \text{for some }j \ \ \text{ where }\  i_j \leq j.
\end{align}
Therefore $\pi(i) = A-n,A,A+n,$ where $A = \pi(i_j-1)-\pi(i_j)+\pi(i-1)$ or  $A =  \pi(i_j)-\pi(i_j-1)+\pi(i-1)$.
Observe that $-(n-1) \leq A \leq 2n$, and only one of the values $A-n,A,A+n$ can be between $1$ and $n$. Hence $\pi(i)$ can be determined \textit{uniquely} from \eqref{sc3} by fixing one of the signs in the equation. This proves the claim. 

Now as $\boldsymbol {\omega}$ is an even word where each distinct letter appears $k_1,k_2,\ldots,k_b$ times (and hence each $k_i$ is even), we first note that by \eqref{even} there are  total $\displaystyle \prod_{i=1}^b {{k_i-1} \choose \frac{k_i}{2}}$ set of equations for determining the non-generating vertices, once the generating vertices are chosen. We have seen from the above argument that only one particular set of equations determine the dependent vertices uniquely. 
Suppose we start with one particular choice of equations. Consider another set of equations. We know that the equations differ only in sign. If the change of sign occurs first at the $m$-th position (i.e. suppose $s_m =    s_{i_j}\ (\text{mod}\ n)$ in the first set,  then in the other set of equations we have, $s_m = -   s_{i_j}\ (\text{mod}\ n)$), then for that $m$, the corresponding value of $A$  changes  because $\pi(m-1)$ remains unaltered and $s_{i_j}$ changes sign in the expression. This changes the value of $\pi(m)$ obtained  from the previous  set of equations. Hence we conclude that each set of equations gives a distinct (unique) choice for $(\pi(0),\pi(1),\ldots,\pi(2k-1),\pi(2k))$. Therefore 
$$\displaystyle{\lim_{n \rightarrow \infty}\frac{1}{n^{b+1}} | \Pi(\boldsymbol {\omega})|= \prod_{i=1}^b {{k_i-1} \choose \frac{k_i}{2}}} \ \ 
\text{if}\ \boldsymbol {\omega} \ \text{is even}.$$ 
\noindent This completes the proof of the lemma.
\end{proof} 
 
\subsection{$\lim_{n\to\infty}\frac{| \Pi(\boldsymbol {\omega})|}{n^{b+1}}$ for Toeplitz Matrix}
\begin{lemma}\label{lem:toe}
	 Suppose $\boldsymbol {\omega}$ is a word with $b$ distinct letters. Then $\displaystyle \lim_{n \rightarrow \infty}\frac{1}{n^{b+1}} | \Pi(\boldsymbol {\omega})|= \alpha(\boldsymbol {\omega})>0$ if and only if $\boldsymbol {\omega}$ is an even word.
\end{lemma}	

\begin{proof}
 Let
\begin{align*} 
s_i= \pi(i)- \pi(i-1)  \  \  \text{for }1 \leq i \leq 2k.
\end{align*}
Clearly, $\boldsymbol {\omega}[i]=\boldsymbol {\omega}[j]$ 
if and only if
$|s_i| =  |s_j|$, that is, $s_i - s_j  =  0$ or   $s_i  +  s_j  =  0.$

Now we fix an $\boldsymbol {\omega}$ with $b$ distinct letters. Suppose $i_1,i_2,\ldots,i_b$ are the positions where new letters made their first appearances. Clearly $s_{i_1}=s_1$. If the first letter appears in the $j$-th position, then $s_j=s_1$ or $s_j=-s_1$. Similarly,  for every $i,\ 1 \leq i \leq 2k$,
\begin{equation}
s_i = s_{i_j} \mbox{ or }	s_i =  - s_{i_j}\label{tplus_minus}
\end{equation} 
for some $j \in \{1,2,\ldots,b\}$. 	  
We split the proof into a few steps.
	
\vskip3pt	
\noindent \textbf{Step 1}.	If $\pi(i_j), 0 \leq j \leq b$ can be chosen freely, then $\{\pi(0),s_{i_j};1 \leq j \leq b\}$ do not satisfy any non-trivial liner equation. Conversely if $\{\pi(0),s_{i_j},1\leq j\leq b\}$ do not any satisfy non-trivial linear equation then  $\pi(i_j),0\leq j\leq b$ can be chosen freely.
This follows from the same argument as in Steps 1 and 2 of the proof of Lemma \ref{lem:sym}.

\vskip3pt

\noindent \textbf{Step 2}. Suppose $\pi(0)$ and $s_{i_j},1\leq j\leq b$ can be chosen freely. Then the word is even.	
	 
To see this, observe that the circuit condition gives
\begin{equation}\label{t2}
 \sum_{i=1}^{k} s_i=\pi(0) - \pi(k)=0.
\end{equation}
Therefore, using \eqref{tplus_minus}, we see that there exists $\alpha_j \in \mathbb{Z}$ for all $1 \leq j \leq b$ such that 
\begin{equation*}
\alpha_1s_{i_1}+ \alpha_2 s_{i_2} + \cdots + \alpha_b s_{i_b} = 0.
\end{equation*}
Since we need the $s_{i_j}$'s to have free choice, we must have $\alpha_j=0$ for all $j \in \{1,2,\ldots,b\}$.	 Therefore for each $j$, 
 \begin{align*}
\big|\{l:s_l=s_{i_j}\}\big|=\big|\{l:s_l=-s_{i_j}\}\big|.
 \end{align*} 
  So each letter appears an even number of times and the word is even. Therefore 
	$$\displaystyle \lim_{n \rightarrow \infty}\frac{1}{n^{b+1}} | \Pi(\boldsymbol {\omega})|=0\ \ \text{if}\ \boldsymbol {\omega}\ \text{is not an even word}.$$   
\noindent \textbf{Step 3}. Suppose $\boldsymbol {\omega}$ is an even word with $b$ distinct letters. 
 First we fix the generating vertices $\pi(i_j),\ 0\leq j\leq b$.  Consider $\pi(i)$. By \eqref{tplus_minus}, we have that
 \begin{align}\label{relation}
 \pi(i)  &=  \pm(\pi(i_j)-\pi(i_j-1))+\pi(i-1)=\pm s_{i_j}+ \pi(i-1) \ \ \text{for some }j.
  \end{align}
  Let 
  $$
  v_i=  \frac{\pi(i)}{n}   \text{ for }0 \le i \leq 2k \ \mbox{ and }\  
   u_j= \frac{s_{j}}{n} \text{ for }1 \leq j \leq 2k.
  $$
 Clearly, $\pi(i)= \pi(i-1) \pm s_{i_j}$ whenever $i$-th letter in $\boldsymbol{\omega}$ is same as the $j$-th distinct letter that appeared first at the $i_j$-th position. Therefore $v_i= v_{i-1} \pm u_{i_j}$. Also observe that $v_1= \ v_0 + u_{i_1}(=u_1)$.
  
Let $$S=\{\pi(i_j):0 \leq  j \leq b\}\ \text{ and }\ S^{\prime}=\{i:\ \pi(i) \notin S\}.$$ 
That is, $S$ is the set of all distinct generating vertices and $S^{\prime}$ is the set of all indices of the non-generating vertices.
We have the following claim.
\vskip3pt  
\noindent \textbf{Claim}: For any $1 \leq i \leq 2k$, $v_i= v_0+  \sum_{j=1}^i \alpha_{ij}u_{i_j}$ where  $\alpha_{ij}$ depends on the choice of sign in  \eqref{relation}.
\vskip3pt
\noindent We prove this by induction. We know that $\pi(1) \in S$. Clearly, $v_1=s_1 +v_0$. Now either $\pi(2) \in S$ or $2 \in S^{\prime}$. If $\pi(2)\in S$, then $v_2=s_2-v_1$ and $v_1=u_1-v_0$. Therefore $v_2= u_2-s_1+u_0$. If  $2\in S^{\prime}$, then
  $u_2= \pm u_1$  and $v_2= v_1 \pm u_1$.
  So either $v_2= u_1 + v_0 + u_1$ or $ v_2= u_1 + v_0 - u_1=v_0$.
Hence the claim is true for $i=2$.

Now we assume that the claim is true for all $j<i$ and try to prove it for $i$.
Then either $\pi(i) \in S$ or $i \in S^{\prime}$. If $\pi(i) \in S$, then 
  \begin{align*}
  v_i&= u_i+v_{i-1}\\
 &= u_i+ v_0+  \sum_{j=1}^{i-1} \alpha_{(i-1)j}u_{i_j}  \   \   \text{(by induction hypothesis)}\\
 &=  v_0 +  \sum_{j=1}^{i} \alpha_{ij}u_{i_j}, 
  \end{align*}
where $\alpha_{ii}=1$.  If $i \in S^{\prime}$, then there exists $j$ such that $i_j<i$ and $u_i=\pm u_{i_j}$. Then either  $ v_i= v_{i-1}+ u_{i_j}$, or     $v_i= v_{i-1}- u_{i_j}.$
Hence either $$ v_i=v_0+ \sum_{j=1}^{i-1} \alpha_{(i-1)j}u_{i_j}  + u_{i_j}, \ \mbox{ or }\  v_i=v_0+  \sum_{j=1}^{i-1} \alpha_{(i-1)j}u_{i_j} - u_{i_j}.$$
  Therefore $v_i=v_0 +  \sum_{j=1}^{i} \alpha_{ij}u_{i_j}$  where  $\alpha_{ij}=\alpha_{(i-1)j}+1$ or $\alpha_{(i-1)j}-1$ (depending on the sign of the equation).  Thus the claim is proved.

Therefore for $1\leq i \leq 2k$, we have 
  \begin{equation*}
  v_i= v_0+ L_{i,u}^T(u_{S}),
  \end{equation*}
  where $L_{i,u}^T(u_{S})$ denotes a linear combination of $\{u_i:\pi(i)\in S\}$.

Also, for $1\leq i \leq 2k$, we have 
  \begin{equation*}
  v_i= v_0+ L_i^T(v_{S}),
  \end{equation*}
  where $L_i^T(v_{S})$ denotes a linear combination of $\{v_i:\pi(i)\in S\}$ arising from \eqref{tplus_minus}. 
 Now, the linear combinations vary depending on the sign chosen for each $s_i$. As we know for each block of an even word, the number of positive and negative signs in the relations among the $s_i$'s (i.e., the equation like \eqref{tplus_minus}) are same. Therefore there are $\displaystyle \prod_{i=1}^b{ {k_i-1} \choose {\frac{k_i}{2}}}$ different sets of linear combinations corresponding to each word $\boldsymbol{\omega}$, where $k_1,\ldots,k_b$ are the block sizes of $\boldsymbol {\omega}$. 
  
Let $U_n=\{0,1/n,\ldots,(n-1)/n\}$. Then from \eqref{Pi(omega)}, it is easy to see that for a word $\boldsymbol{\omega}$ of length $2k$, 
\begin{align*}
\big|\Pi(\boldsymbol{\omega})\big|= \big|\big\{& (v_0,v_1,\ldots,v_{2k}): v_i\in U_n \text{ for } 0\leq i\leq 2k, v_0=v_{2k},\\
& L(v_{i-1},v_i)=L(v_{j-1},v_j) \text{ whenever } \boldsymbol{\omega}[i]=\boldsymbol{\omega}[j]\big \}\big|.
 \end{align*}
Hence
\begin{equation}\label{integral-gen}
  \displaystyle \lim_{n \rightarrow \infty}\frac{1}{n^{b+1}} |\Pi(\boldsymbol{\omega})|= \sum_{L_{\boldsymbol{\omega}}^T}\int_{0}^{1} \int_{0}^{1}\int_{0}^{1} \cdots  \int_{0}^{1} \textbf{1}(0 \leq v_0+ L_i^T(v_S) \leq 1,\ \forall \  i \in S^{\prime})\ dv_S,
\end{equation}
where $dv_S= \displaystyle \prod_{j=0}^{b}dv_{i_j}$  denotes the $(b+1)$-dimensional Lebesgue measure, $v_{i_0}=v_0$ and $\sum_{L_{\boldsymbol{\omega}}^T}$ is the sum over all the $\displaystyle \prod_{i=1}^b{ {k_i-1} \choose {\frac{k_i}{2}}}$ such different sets of linear combinations corresponding to $\boldsymbol{\omega}$.

As observed in Step 1, choosing $v_{i_j}, 0 \leq j \leq b$ freely is equivalent to choosing  $v_0$ and $u_{i_j},1 \leq j \leq b$ freely. So
  \begin{equation}\label{integral-exp}
  \displaystyle \lim_{n \rightarrow \infty}\frac{1}{n^{b+1}} |\Pi(\boldsymbol{\omega})|= \sum_{L_{\boldsymbol{\omega}}^T}\int_{0}^{1} \int_{-1}^{1}\int_{-1}^{1} \cdots  \int_{-1}^{1} \textbf{1}(-1 \leq v_0+ L_{i,u}^T(u_S) \leq 1,\ \forall \  i \in S^{\prime})\ du_S,
\end{equation}    
where $du_S= \displaystyle \prod_{j=0}^{b}u_{i_j}$  denotes the $(b+1)$-dimensional Lebesgue measure on $[0,1]\times [-1,1]^{b}$ and $u_{i_0}=v_0$.

  Suppose, a particular set of linear combinations $L_{i,u}^T$ is given,
  i.e., for $i \in S^{\prime}$, $v_i=v_0 + \displaystyle \sum_{m=1}^{i} \alpha_{im}u_{i_m}$ and the values of $\alpha_{im},1 \leq i \leq 2k, 1\leq j \leq b$ are known.
 Then we choose $$C= \max \{|\alpha_{ij}|: 1 \leq j \leq b \text{ and } i \in S^{\prime} \}.$$
 Next we choose $\epsilon$ such that  $Cb\epsilon  <  1/2$. 
 Now, let $|u_{i_j}|< \epsilon$ for $1 \leq j \leq b$ and $Cb\epsilon \leq v_0 \leq 1-Cb\epsilon$. 
 Then, for all $i \in S^{\prime}$, $0 \leq v_0+ L_{i,u}^T (u_S) \leq 1$.
  Also the circuit condition is automatically satisfied as we have \eqref{t2}. Note that  we cannot choose the $s_{i_j}$'s freely in case of words that are not even as observed in Step 2.
  
 Thus
$$\displaystyle \lim_{n \rightarrow \infty}\frac{1}{n^{b+1}} |\Pi(\boldsymbol {\omega})|=\alpha(\boldsymbol{\omega})> 0 \ \ \text{for any even word}\ \boldsymbol{\omega},$$ 
where $\alpha(\boldsymbol{\omega})$ is the sum of the of the integrals defined in \eqref{integral-exp}.

 This completes the proof of the lemma.\end{proof}

\subsection{$\lim_{n\to\infty}\frac{| \Pi(\boldsymbol {\omega})|}{n^{b+1}}$ for Hankel Matrix}	

\begin{lemma}\label{lem:han}
Suppose $\boldsymbol {\omega}$ is a word with $b$ distinct letters. Then $\displaystyle \lim_{n \rightarrow \infty}\frac{1}{n^{b+1}} | \Pi(\boldsymbol {\omega})|= \alpha(\boldsymbol{\omega})>0$ if and only if $\boldsymbol {\omega}$ is a symmetric word. Moreover, for every symmetric word $\boldsymbol{\omega}$, $0< \alpha(\boldsymbol{\omega}) \leq 1$.
\end{lemma}

\begin{proof}
Let 
\begin{align*}
t_i= \pi(i)+ \pi(i-1)  \text{ for } 1 \leq i \leq 2k.
\end{align*}
 Clearly, $\boldsymbol {\omega}[i]=\boldsymbol {\omega}[j]$ 
 if and only if $\pi(i-1) +  \pi(i)  =  \pi(j-1) +  \pi(j)$, that is, $t_i -  t_j = 0$.

Now we fix an $\boldsymbol {\omega}$ with $b$ distinct letters. Suppose $i_1,i_2,\ldots,i_b$ are the positions where new letters made their first appearances. 
Clearly $t_{i_1}=t_1$. If the first letter again appears at $j$-th position, then $t_j = t_1$. Similarly,  for every $i,\ 1 \leq i \leq 2k$, 
\begin{equation}\label{han1}
t_i = t_{i_j}\ \ \text{ for some }j \in \{1,2,\ldots,b\}.
\end{equation}

\noindent \textbf{Step 1}. If $\pi(i_j),0 \leq j \leq b$ are chosen freely, then $\{\pi(0),t_{i_j};1\leq j\leq b\}$ do not satisfy any non-trivial linear relation. Conversely if $\pi(0),t_{i_j},1\leq j\leq b$ satisfy non-trivial linear relation then $\pi(i_j),0\leq j\leq b$ can be chosen freely. This follows from the same argument as in Steps 1 and 2 of the proof of Lemma \ref{lem:rev}.

\vskip3pt
\noindent \textbf{Step 2}. Suppose $\pi(0)$ and $t_{i_j}, 1 \leq j \leq b$, can be chosen freely. Then the word is symmetric.

To see this,  observe that if the length of the word is $2k$, then
\begin{equation*}
  (t_1+t_3+\cdots+t_{2k-1}) -  (t_2+t_4+\cdots+t_{2k})= \pi(0)\ - \pi(2k)=0.
 \end{equation*}
 Hence we have 
 \begin{equation}\label{han2}
(t_1+t_3+\cdots+t_{2k-1}) -  (t_2+t_4+\cdots+t_{2k})= 0.
\end{equation}
Therefore, from \eqref{han1}, we see that there exists $\alpha_j \in \mathbb{Z}$ for all $1 \leq j \leq b$ such that \eqref{han2} can be written as 
\begin{equation*}
\alpha_1t_{i_1}+ \alpha_2 t_{i_2} + \cdots + \alpha_b t_{i_b} =0.
		 \end{equation*}
	Now for $t_{i_j}$'s to have free choice, we must have $\alpha_j=0$ for all $j \in \{1,2,\ldots,b\}$.
		 Thus, for each $j$,
\begin{align*}
\big|\big\{l:l \text{ odd and }t_l=t_{i_j} \big\}\big|=\big|\big\{l:l \text{ even and }t_l=t_{i_j}\big\}\big|.
\end{align*} 
That is, each letter appears equal number of times at odd and even places. Hence the word is symmetric. 

 Now if the length of the word is odd, say $2k+1$, then $$\pi(2k+1)-\pi(0)=(t_1+t_2+\cdots+t_{2k+1})-(t_2+t_4+\cdots+t_{2k})-2\pi(0)=0.$$ Hence substituting $t_i$ by $t_{i_j}$'s using \eqref{rc1} in the above equation, we see that $\pi(0),t_{i_1},\ldots,t_{i_b}$ satisfies a non-trivial linear relation. Hence the length of the word cannot be odd.

Therefore if $\boldsymbol {\omega}$ is not symmetric, at least one of the generating vertices cannot be chosen freely and hence
$$\displaystyle \lim_{n \rightarrow \infty}\frac{1}{n^{b+1}} | \Pi(\boldsymbol {\omega})|=0\ \ \text{if}\ \boldsymbol {\omega}\  \text{is not symmetric}.$$
\noindent \textbf{Step 3}. We now show that, if $\boldsymbol {\omega}$ is a symmetric word with $b$ distinct letters, then \\ $\displaystyle \lim_{n \rightarrow \infty}\frac{1}{n^{b+1}} | \Pi(\boldsymbol {\omega})|=\alpha(\boldsymbol {\omega})>0$.
 
 First we fix the generating vertices $\pi(i_j),\ j=0,1,2,\ldots,b$.
 Let 
  $$
  v_i=  \frac{\pi(i)}{n}  \text{ for }0 \le i \leq 2k,\ \ S=\{\pi(i_j):0 \leq  j \leq b\}  \ \mbox{ and }\ S^{\prime}=\{i : \pi(i) \notin S\}.$$
For $1\leq i \leq 2k$, from the link function and the formula for $t_i$ we have 
  \begin{equation}\label{linear-han}
  v_i= L_{i}^H(v_{S}),
  \end{equation}
  where $L_{i}^H(v_{S})$ denotes a linear combination of $\{v_i:\pi(i)\in S\}$.  
  
 Let $U_n=\{0,1/n,\ldots,(n-1)/n\}$. From \eqref{Pi(omega)}, it is easy to see that for a word $\boldsymbol{\omega}$ of length $2k$, 
\begin{align*}
\big|\Pi(\boldsymbol{\omega})\big|= \big|\big\{ (v_0,v_1,\ldots,v_{2k}): v_i\in U_n \text{ for } 0\leq i\leq 2k, v_0=v_{2k},
v_i= L_i^H(v_S)\big \}\big|.
 \end{align*}
Transforming $v_i \mapsto y_i= v_i -\frac{1}{2}$, we get that 
\begin{align*}\big|\Pi(\boldsymbol{\omega})\big|= \big|\big\{ (y_0,y_1,\ldots,y_{2k}): & \ y_i\in \{-1/2,-1/2+1/n, \ldots,-1/2+(n-1)/n\} \text{ for } 0\leq i\leq 2k,\\
&\quad y_0=y_{2k} \mbox{ and }
y_i= L_i^H(y_S)\big \}\big|.\end{align*}  
%
Hence
 \begin{equation}\label{han4}
  \displaystyle \lim_{n \rightarrow \infty}\frac{1}{n^{b+1}}\Pi(\boldsymbol{\omega})= \int_{-1/2}^{1/2} \int_{-1/2}^{1/2}\int_{-1/2}^{1/2} \cdots  \int_{-1/2}^{1/2} \textbf{1}(-1/2 \leq L_i^H(y_S) \leq 1/2,\ \forall\ i \in S^{\prime})\ dy_S,
\end{equation}
where $dy_S= \displaystyle \prod_{j=0}^{b}dy_{i_j}$  denotes the $(b+1)$-dimensional Lebesgue measure on $[-\frac{1}{2},\frac{1}{2}]^{b+1}$.

Let 
$p_i=y_{i-1}+y_i$ and $q_i=y_{i-1}-y_i$. Now we have the following claim.

\vskip3pt  
\noindent \textbf{Claim}: For any $1\leq i \leq 2k$,
$$y_i= \begin{cases}
y_0+  \sum_{j=1}^{i} \alpha_{ij}p_{i_j} & \text{ if } \ i \ \mbox{ is  even},\\
 -y_0+  \sum_{j=1}^{i} \alpha_{ij}p_{i_j} & \text{ if }\  i \  \text{ is odd}.
\end{cases} $$
\vskip3pt
\noindent We prove this by induction. 
We know that $\pi(1) \in S$. Clearly, $y_1=p_1 -y_0$. Now either $\pi(2) \in S$ or $2 \in S^{\prime}$. If $\pi(2)\in S$, then $y_2=p_2-y_1$. Therefore $y_2= p_2-p_1+y_0$. If  $2\in S^{\prime}$, then
  $p_2=  p_1$  and $y_2= p_1 -y_1=y_0$.
So the claim is true for $i=2$.
  
  Now we assume that the claim is true for all $j<i$ and try to prove it for $i$.
Then either $\pi(i) \in S$ or $i \in S^{\prime}$. 

If $\pi(i) \in S$, then $y_i= p_i-y_{i-1}$. If $i$ is even, then $i-1$ is odd and hence $y_{i-1}= -y_0+ \displaystyle \sum_{j=1}^{i-1}\alpha_{(i-1)j}p_{i_j}$ by induction hypothesis. Therefore $y_i= y_0+ \displaystyle \sum_{j=1}^{i}\alpha_{ij}p_{i_j}$ where $\alpha_{ii}=1$. The case where $i$ is odd can be tackled similarly.

   If $i \in S^{\prime}$, then there exists $m$ such that $i_m<i$ and $p_i=p_{i_m}$. Then $y_i=p_{i_m}-y_{i-1}$. Now if $i$ is even, then $i-1$ is odd and $y_{i-1}= -y_0+ \displaystyle \sum_{j=1}^{i-1}\alpha_{(i-1)j}p_{i_j}$ by induction hypothesis. Therefore $y_i= y_0+ \displaystyle \sum_{j=1}^{i-1}\alpha_{ij}p_{i_j}$ where $\alpha_{im}=\alpha_{(i-1)m}+1$. The case where $i$ is odd can be tackled similarly.
   
Thus the claim is proved.
\vskip3pt
%
%
Now we perform the following change of variables in \eqref{han4}:
\begin{align*}
(y_0,y_1,y_2,y_3,\ldots,y_{2k}) \longrightarrow (y_0,-y_1,y_2,-y_3,\ldots,y_{2k})=(z_0,z_1,z_2,z_3,\ldots,z_{2k}) \ (\text{ say }) .
\end{align*}
Observe that this transformation does not alter the value of the integral in \eqref{han4}. Also under this transformation, 
\begin{align*}
(p_1,p_2,p_3,\ldots,p_{2k}) \longrightarrow (q_1,-q_2,q_3,\ldots,-q_{2k}).
\end{align*}
%
 Then from the claim it follows that $$z_i= z_0+\displaystyle \sum_{j=1}^{i}\beta_{ij}q_{i_j},$$ where $\beta_{ij}=\pm \alpha_{ij}$ according as $i_j$ is odd or even.  We shall use the notation $z_i=L_{i,q}^H(z_S)$ to denote this linear relation.

Also note that from Step 1 choosing $y_{i_j}, 0 \leq j \leq b$ freely is equivalent to choosing $p_{i_j}, 0 \leq j \leq b $ (where $p_{i_0}=y_0$). Therefore we can write \eqref{han4} as 
\begin{align*}
 \displaystyle \lim_{n \rightarrow \infty}\frac{1}{n^{b+1}}\Pi(\boldsymbol{\omega})= \int_{-1/2}^{1/2} \int_{-1}^{1}\int_{-1}^{1} \cdots  \int_{-1}^{1} \textbf{1}(-1/2 \leq L_i^H(z_S) \leq 1/2,\ \forall\ i \in S^{\prime})\ dq_S,
\end{align*}
where $dq_S= \displaystyle \prod_{j=0}^{b}dq_{i_j}$  denotes the $(b+1)$-dimensional Lebesgue measure on $[-\frac{1}{2},\frac{1}{2}]\times [-1,1]^{b}$.

Let $$C= \max \{|\alpha_{ij}|: 1 \leq j \leq b \text{ and } i \in S^{\prime} \}.$$
 Next we choose $\epsilon$ such that  $Cb\epsilon  <  1/4$. 
 Now, let $|q_{i_j}|< \epsilon$ for $1 \leq j \leq b$
 and $Cb\epsilon -\frac{1}{2}\leq z_0 \leq \frac{1}{2}-Cb\epsilon$.
 Then, for all $i \in S^{\prime}$, $-\frac{1}{2} \leq L_{i,q}^T (z_S) \leq \frac{1}{2}$.
  Also the circuit condition is automatically satisfied.

 Hence 
$$\lim_{n \rightarrow \infty}\frac{1}{n^{b+1}} | \Pi(\boldsymbol {\omega})|= \alpha(\boldsymbol{\omega})>0,$$ where $\alpha(\boldsymbol{\omega})$ is the value of the integral in \eqref{han4}. Also as $\alpha(\boldsymbol{\omega})$ actually gives Lebesgue measure of a set in $[-1/2,1/2]^{b+1}$ as seen in \eqref{han4}, we have $\alpha(\boldsymbol{\omega}) \leq 1$.

This completes the proof of the lemma.
\end{proof}

\section{Proofs of the theorems}\label{proofs of the theorems}
We will use moment method to prove our theorems. But first, we recall the $d_2$ metric which will helps us to deal with situations where entries may not have zero means and/or entries which may require truncation to ensure finiteness of all moments.

Let $F$ and $G$ be two distributions with finite second moment. Then the $d_2$ distance between them is defined as
\begin{align*}
d_2(F,G)= \Big[ \underset{(X \sim  F, Y \sim G)}{\inf} \mathbb{E}[X-Y]^2\Big]^{\frac{1}{2}},
\end{align*}
where $(X\sim F, Y \sim G)$ denotes that the joint distribution of $(X,Y)$ is such that the marginal distributions are $F$ and $G$ respectively.

It is well-known that if $d_2(F_n,F) \rightarrow 0$ as $n \rightarrow \infty$, then $F_n \xrightarrow {\mathcal{D}} F$.
For a proof of this result, see Lemma 1.3.1 in \citep{bose2018patterned}.

The following lemma will be useful to apply the above result on the EESD of the matrices.
\begin{lemma}\label{lem:d2}
Let $\{F_i\}_{1\leq i\leq n}$ and $\{G_i\}_{1\leq i\leq n}$ be two sequence of distributions with finite second moment. Suppose $\{X_i\}_{1\leq i\leq n}$ and $\{Y_i\}_{1\leq i\leq n}$ are sequences of random variables such that $X_i\sim F_i$ and $Y_i\sim G_i$. Then 
\begin{equation}
d_2^2(\frac{1}{n} \displaystyle \sum_{i=1}^nF_i,\frac{1}{n} \displaystyle \sum_{i=1}^nG_i) \leq \displaystyle \frac{1}{n}\sum_{i=1}^n \mathbb{E}[(X_i-Y_i)^2].
\end{equation}
\end{lemma}
\begin{proof}
Let $e_k$ denote the $k-$th canonical vector in $\mathbb{R}^{n}$, i.e, the vector whose $k-$th coordinate is 1 and all other coordinates are zero. Let $Z_1,Z_2,\ldots,Z_{n}$ be random variables independent of $X_i,Y_i, 1 \leq i \leq n$ such that $(Z_1,Z_2,\ldots,Z_{n})=e_k, $ with probability $1/n$ for each $1\leq k \leq n$.

Let
\begin{align}\label{X-Y}
X  = \displaystyle \sum_{i=1}^{n} Z_iX_i \ \ \text{ and }
Y = \displaystyle \sum_{i=1}^{n} Z_iY_i 
\end{align}
Then $X \sim \frac{1}{n} \displaystyle \sum_{i=1}^nF_i$ and $Y \sim \frac{1}{n} \displaystyle \sum_{i=1}^nG_i$. Thus from the definition of $d_2$,
\begin{equation}\label{d2}
d_2^2(\frac{1}{n} \displaystyle \sum_{i=1}^nF_i,\frac{1}{n} \displaystyle \sum_{i=1}^nG_i) \leq \mathbb{E}[(X-Y)^2].
\end{equation}
Now by \eqref{X-Y}, it is clear that $$\mathbb{E}[(X-Y)^2]= \frac{1}{n}\sum_{i=1}^n \mathbb{E}[(X_i-Y_i)^2].$$
Hence from \eqref{d2} the result follows.
\end{proof}
Now, we prove a lemma that helps us determine the closeness of the EESDs of to random matrices.  
\begin{lemma}\label{lem:metric}
Suppose $A$ and $B$ are two $n \times n$ symmetric matrices and $\mu_A$ and $\mu_B$ are their ESDs. Then
\begin{equation}
d_2^2(\mathbb{E} \mu_A,\mathbb{E} \mu_B) \leq \frac{1}{n}\mathbb{E}[\Tr(A-B)^2].
\end{equation}
\end{lemma} 
\begin{proof}
Let $\{\lambda_i(A)\}_{1 \leq i \leq n}$ and $\{\lambda_i(B)\}_{1 \leq i \leq n}$ denote the eigenvalues of $A$ and $B$, both written in descending order. Then from the definition of EESD and Lemma \ref{lem:d2},
\begin{align*}
d_2^2(\mathbb{E} \mu_A,\mathbb{E} \mu_B) \leq \frac{1}{n}\displaystyle \sum_{i=1}^n \mathbb{E}[(\lambda_i(A)-\lambda_i(B))^2]. 
\end{align*}
Now, using the \textit{Hoffman-Wielandt inequality} (\citep{10.1215/S0012-7094-53-02004-3}), we obtain
\begin{align*}
d_2^2(\mathbb{E} \mu_A,\mathbb{E} \mu_B) \leq \frac{1}{n}\displaystyle \sum_{i=1}^n \mathbb{E}[(\lambda_i(A)-\lambda_i(B))^2] \leq \frac{1}{n}\mathbb{E}[\Tr(A-B)^2]. 
\end{align*}
\end{proof}

 Now we prove our main theorems.
 
\subsection{Proof of Theorem \ref{thm:mainrev}} 
\begin{proof}
[\textbf{Proof of Theorem \ref{thm:mainrev}}]
We separate the proof of the theorem into a few steps.
\vskip3pt

\noindent \textbf{Step 1}. We reduce the general case to the case  where all the entries of $Z_n$ have mean 0.

To see this, consider the matrix $\widetilde{Z}_n$ whose entries are $(y_{i}-\mathbb{E}y_{i})$. Clearly the entries of $\widetilde{Z}_n$ have mean 0. Now
\begin{align}\label{meanzero-noiid}
n\ \mathbb{E}[(y_{i}- \mathbb{E} y_{i})^{2k}]= n\mathbb{E}[y_{i}^{2k}] + n  \sum_{j=0}^{2k-1} {{2k} \choose {j}}\mathbb{E}[y_{i}^{j}] (\mathbb{E}y_{i})^{2k-j}.  
\end{align}
The first term of the r.h.s. is equal to $g_{2k,n}(i/n)$ by \eqref{gkeven}. For the second term we argue as follows:
\begin{align*}
\text{For } j\neq {2k-1}, \ \ n \ \mathbb{E} [y_{i}^{j}](\mathbb{E}y_{i})^{2k-j}& = (n^{\frac{1}{2k-j}} \mathbb{E}y_{i})^{2k-j}  \mathbb{E}[y_{i}^{j}]\\
&\overset{n \rightarrow \infty}{\longrightarrow} 0, \ \ \ \ \text{ by condition } \eqref{gkodd}.  
\end{align*}
\begin{align*}
\text{For } j={2k-1}, \ \ n \ \mathbb{E} [y_{i}^{2k-1}] \mathbb{E}y_{i}& = (\sqrt{n}  \mathbb{E} [y_{i}^{2k-1}])  (\sqrt{n}  \mathbb{E}y_{i})\\
&\overset{n \rightarrow \infty}{\longrightarrow} 0, \ \ \ \ \text{ by condition } \eqref{gkodd}.
\end{align*}
Hence from \eqref{meanzero-noiid}, we see condition \eqref{gkeven} is true for the matrix $\widetilde{Z}_n$. Similarly we can show that \eqref{gkodd} is true for $\widetilde{Z}_n$. Hence Assumption A holds for the matrix $\widetilde{Z}_n$.

Now observe that using Lemma \ref{lem:metric},
\begin{align*}
d_2^2(\mathbb{E}\mu_{Z_n},\mathbb{E}\mu_{\widetilde{Z_n}}) & \leq \frac{1}{n} \displaystyle \sum_{i} \ n (\mathbb{E}y_{i})^2\\ 
& \leq n \ (\underset{i}{\sup} \ \mathbb{E}y_{i})^2 \\
&= (\underset{i}{\sup} \ \sqrt{n}\ \mathbb{E}y_{i})^2 \  \overset{n \rightarrow \infty}{\longrightarrow} 0, \ \ \ \ \text{ by condition } \eqref{gkodd}. 
\end{align*}
Hence the LSD of $Z_n$ and $\widetilde{Z}_n$ are same. Hence we can assume the entries of $Z_n$ have mean zero.

\vskip3pt

Now we prove the part (a) of the theorem.
\vskip3pt

 \noindent \textbf{Step 2}. In this step we verify that for every $k\geq 1$, $\frac{1}{n}\mathbb{E}[\Tr (Z_n)^k] \rightarrow \alpha_k$ as $n \rightarrow \infty$. This is referred to as the first moment condition.
 
From \eqref{momentf1usual}, using the fact that $\mathbb{E}(y_{i})=0$,  we have 	
\begin{equation}\label{momentnoniid}
\lim_{n \rightarrow \infty}\frac{1}{n}\mathbb{E}[\Tr(Z_n)^{k}] =\displaystyle\lim_{n \rightarrow \infty}\frac{1}{n} \sum_{\pi:\ell(\pi)=k}\mathbb{E}[Y_{\pi}] 
 = \lim_{n \rightarrow \infty}  \sum_{b=1}^k  \sum_{\underset{\text{with b distinct letters}}{\omega \ \text{matched}}} \frac{1}{n} \sum_{\pi \in \Pi(\boldsymbol{\omega})} \mathbb{E}(Y_{\pi}).  
\end{equation}
Now suppose $\boldsymbol {\omega}$ is a word with $b$ distinct letters each letter appearing $k
_1,k_2,\ldots,k_b$ times. Let the $j$-th distinct letter appear for the first time at $i_j, 1\leq j \leq b$. 
 Write $(\pi(i_j-1),\pi(i_j))$ as $(m_j,l_j)$. Clearly, $m_1=\pi(0)$ and $l_1=\pi(1)$. 
Let $S$ be the set of distinct generating vertices of $\boldsymbol {\omega}$. 
Suppose $\boldsymbol {\omega}\in E_b(k)$. 
Then the contribution of this $\boldsymbol {\omega}$ in \eqref{momentnoniid} can be written as 
\begin{align}\label{finitesum-rev}
\frac{1}{n^{b+1}}\displaystyle \sum_S \prod_{j=1}^b g_{k_j,n}\bigg(\frac{m_j+l_j-2\ \ (\text{mod }n)}{n}\bigg). 
\end{align}
Now from Step 4 of Lemma \ref{lem:rev}, observe that for $j\neq 1$, $m_j$ can be written as a linear combination of $\{l_i; 1\leq i\leq j-1\}$ and $m_1$. By abuse of notation let $m_1$ and $l_j, 1\leq j\leq b$ denote the indices of the generating vertices. Then, as $n \rightarrow \infty$, the above sum goes to 
\begin{align}\label{limit-rev}
\int_0^1 \int_0^1 \cdots \int_0^1 \prod_{j=1}^b \bigg[g_{k_j}(x_{m_j}+x_{l_j}) \boldsymbol{1}(0 \leq x_{m_j}+x_{l_j}\leq 1)+ g_{k_j}(x_{m_j}+x_{l_j}-1) \boldsymbol{1}(x_{m_j}+x_{l_j}> 1) \bigg] \ dx_S,
\end{align}
where 
$dx_S= dx_{m_1}dx_{l_1}dx_{l_2}\cdots dx_{l_b}$ is the $(b+1)$-dimensional Lebesgue measure on $[0,1]^{(b+1)}$.

By Lemma \ref{lem:rev}, it follows that the above integral is over all of $[0,1]^{(b+1)}$ if and only if $\boldsymbol{\omega}\in S_b(k)$. That is the integral actually reduces to a $c$  dimensional integral where $c \leq b$, if $\boldsymbol{\omega} \notin S_b(k)$ as $x_{m_1}$ and the $x_{l_j}$'s then satisfy a linear equation (see proof of Lemma \ref{lem:rev}).

As a result, the contribution of $\boldsymbol {\omega}$ as described in \eqref{limit-rev} is equal to 0 if $\boldsymbol {\omega}  \in E_b(k)\setminus S_b(k)$.
If $\boldsymbol {\omega} \in S_b(k)$, then the contribution of $\boldsymbol {\omega}$ is  
\begin{align}\label{limit-revf}
\int_0^1 \int_0^1 \cdots \int_0^1 \prod_{j=1}^b h_{k_j}(x_{m_j},x_{l_j}) \ dx_S,
\end{align}
where $h_{k_j}(x_{m_j},x_{l_j})=g_{k_j}(x_{m_j}+x_{l_j})   \boldsymbol{1}(0 \leq x_{m_j}+x_{l_j}\leq 1)+ g_{k_j}(x_{m_j}+x_{l_j}-1) \boldsymbol{1}(x_{m_j}+x_{l_j}> 1)$.

Now for any $m \in \mathbb{N}$, let 
$$f_{2m}(x)=\int_{0}^1 h_{2m}(x,y) \ dy = \int_{0}^1 g_{2m}(x+y) \boldsymbol{1}(0 \leq x+y \leq 1)+ g_{2m}(x+y-1) \boldsymbol{1}( x+y >1)\ dy .$$
Note that 
\begin{align}\label{integral-rev}
f_{2m}(x)&= \int_x^1 g_{2m}(t) \ dt \ + \int_0^x g_{2m}(t) \ dt= \int_0^1 g_{2m}(t) \ dt =C_{2m},
\end{align}
is independent of $x$.
Then, using \eqref{integral-rev}, \eqref{limit-revf} can be written as 
\begin{align}
{\int_0^1 \int_0^1 \cdots \int_0^1} \prod_{j=1}^{b-1} h_{k_j}(x_{m_j},x_{l_j})C_{s_b} \ dx_{m_1} dx_{l_1}\cdots dx_{l_{b-1}}.
\end{align}
Proceeding in this manner  the contribution of $\boldsymbol {\omega}$ from  \eqref{limit-revf} can be obtained as follows:
\begin{align}\label{limit-revfi}
\int_0^1 \prod_{j=1}^b C_{k_j} \ dx_0= \prod_{j=1}^b C_{k_j}.
\end{align}
Now suppose $\boldsymbol {\omega} \notin E(2k)$. Suppose $\boldsymbol{\omega}$ contains $b_1$ distinct letters that appears even number of times and $b_2$ number of distinct letters that appears odd number of times and $b=b_1+b_2$. So
we  assume that for each $\pi \in \Pi(\boldsymbol {\omega})$, 
$k_{j_p}$, $1\leq p \leq b_1$ are even and $k_{j_q}$, $b_1+1\leq q \leq b_1+b_2$ are odd. Hence the contribution of this $\boldsymbol {\omega}$ to  \eqref{momentnoniid} is as follows:
\begin{align}\label{finitesum-nons2k}
&\frac{1}{n}   n^{-{b_1}} n^{-(b_2-\frac{1}{2})}
\sum_{S} \prod_{p=1}^{b_1} h_{k_{j_p}}(x_{m_{j_p}},x_{l_{j_p}}) \prod_{q=b_1+1}^{b_1+b_2} n^{\frac{b_2-1/2}{b_2}} \mathbb{E}\Big[y_{(t_{i_q}-2) (\text{mod }n)}^{k_{j_q}}\Big] \nonumber\\
&= \frac{1}{n^{b_1+b_2+\frac{1}{2}}} \sum_{S} \prod_{p=1}^{b_1} h_{k_{j_p}}(x_{m_{j_p}},x_{l_{j_p}})   \prod_{q=b_1+1}^{b_1+b_2} n^{\frac{b_2-1/2}{b_2}} \mathbb{E}\Big[y_{(t_{i_q}-2)\ (\text{mod }n)}^{k_{j_q}}\Big].
\end{align}
For $n$ large, $n^{\frac{b_2-1/2}{b_2}} \mathbb{E}[y_{(t_{i_q}-2)\ (\text{mod }n)}^{k_{j_q}}]<1$ for any $b_1+1\leq q\leq b_1+b_2$ and $\prod_{p=1}^{b_1} h_{k_{j_p}}(x_{m_{j_p}},x_{l_{j_p}})\leq M$ (independent of $n$). Now as $\boldsymbol{\omega} \notin S_b(k)$, from Lemma \ref{lem:rev} we have, $|S|\leq b$. Thus any word that is not even  contributes 0 as $n \rightarrow \infty$.

For any partition $\sigma\in S_b(2k)$, let $\{V_1, \ldots V_b\}$ be its partition blocks. Then from \eqref{momentnoniid} and \eqref{limit-revfi}, we have 
\begin{align}\label{rev-noniid}
\lim_{n \rightarrow \infty}\frac{1}{n}\mathbb{E}[\Tr(Z_n)^{2k}]& =    \sum_{b=1}^{k} \sum_{\sigma \in S_b(2k)} \prod_{i=1}^b C_{|V_i|}= \displaystyle \sum_{\sigma \in S(2k)} C_{\sigma}.
 \end{align}
We also note that $\displaystyle\lim_{n \rightarrow \infty}\frac{1}{n}\mathbb{E}[\Tr(Z_n)^{2k+1}]=0$ for any $k\geq 0$. 
 This proves the first moment condition.
\vskip3pt

\noindent \textbf{Step 3}. Here we show that $ \lim_{n \rightarrow \infty}\frac{1}{n}\mathbb{E}[\Tr(Z_n)^{2k}]=\gamma_{2k}$ determines a unique distribution.
\begin{align*}
\gamma_{2k}= &\lim_{n \rightarrow \infty}\frac{1}{n}\mathbb{E}[\Tr(Z_n)^{2k}]
 \leq \displaystyle \sum_{\sigma \in S(2k)} M_{\sigma}
 \leq \displaystyle \sum_{\sigma \in \mathcal{P}(2k)} M_{\sigma} 
 =\alpha_{2k}.
\end{align*}
As $\{\alpha_{2k}\}$ satisfies Carleman's condition, $\{\gamma_{2k}\}$  does so. Hence the sequence of moments $\{\gamma_{2k}\}$ determines a unique distribution.

Therefore, there exists a measure $ \mu $ with moment sequence $\{\gamma_{2k}\}$ such that $\mathbb{E}\mu_{Z_n}$ converges weakly to $ \mu$.

It is easy to see by the explicit  formula for eigenvalues of $RC_n$ (see (1.3) of \citep{bose2018random}) that the measures $\mathbb{E}\mu_{Z_n}$ are symmetric for every $n\geq 1$. Hence the LSD $\mu$ is symmetric and $\{0,C_2,0,C_4,\ldots\}$ gives the sequence of half cumulants for it. 

\noindent This completes the proof of part (a).

\vskip3pt
\noindent \textbf{Step 4}. To prove part (b) of the theorem, observe that from Lemma \ref{lem:metric},
 \begin{align}\label{d2inequality}
 d_2^2(\mathbb{E}\mu_{RC_n},\mathbb{E}\mu_{Z_n}) \leq \frac{1}{n} \mathbb{E}[ \Tr(RC_n-Z_n)^2] =
  \frac{1}{n} \sum_{i} n \mathbb{E}[ x_{i}^2[\boldsymbol {1}_{[|x_{i}| > t_n]}],
 \end{align}
 where the last equality follows as each $x_i$ occurs $n$ times in $RC_n$.
 
Now if $\{t_n\}$ also satisfies condition \eqref{truncation}, then using \eqref{d2inequality} and part (a) we can say that the EESD of $RC_n$ converges to $ \mu$. This proves part (b).
\end{proof}

\subsection{Proof of Theorem \ref{thm:mainsym}}
First of all note that the entries of the matrix $Z_n$ can be assumed to have mean zero. This reduction follows from Step 1 in the proof of Theorem \ref{thm:mainrev}. 


\vskip3pt

Next we prove the first moment condition. From \eqref{momentf1usual} and using the fact that $\mathbb{E}(y_{i})=0$,  we have 	
\begin{align}\label{momentnoniid-sym}
\lim_{n \rightarrow \infty}\frac{1}{n}\mathbb{E}\big[\Tr(Z_n)^{k}\big]& =\lim_{n \rightarrow \infty}\frac{1}{n} \sum_{\pi:\ell(\pi)=k}\mathbb{E}[Y_{\pi}] 
 = \lim_{n \rightarrow \infty}  \sum_{b=1}^k  \sum_{\underset{\text{with b distinct letters}}{\omega \ \text{matched}}}\frac{1}{n} \sum_{\pi \in \Pi(\boldsymbol{\omega})} \mathbb{E}[Y_{\pi}].  
\end{align}
Now suppose $\boldsymbol {\omega}$ is a word with $b$ distinct letters each letter appearing $k
_1,k_2,\ldots,k_b$ times. 
Let the $j$-th distinct letter appears at $(\pi(i_j-1),\pi(i_j))$-th position for the first time. Denote $(\pi(i_j-1),\pi(i_j))$ as $(m_j,l_j)$. 
Let $u_i=s_{i}/n$ as defined in Lemma \ref{lem:sym} and $U_n= \{0,1/n,2/n,\ldots,(n-1)/n\}$. 

Let $\boldsymbol {\omega}\in E_b(k)$ where each distinct letter appears $k_1,k_2,\ldots,k_b$ times. Clearly as observed in Lemma \ref{lem:sym} there are $ \prod_{i=1}^b {{k_i-1} \choose \frac{k_i}{2}}$ equations for determining the non-generating vertices, once the generating vertices are chosen. For each of the combination of equations we get the same contribution due to the structure of the link function. Then the contribution of each such combination of equations for the word $\boldsymbol {\omega}$ in \eqref{momentnoniid} can be written as 
\begin{align}\label{finitesum-sym}
\frac{1}{n^{b+1}}\displaystyle \sum_S \prod_{j=1}^b g_{k_j,n}\bigg(\frac{1}{2}- |\frac{1}{2}-|u_{i_j}||\bigg),  
\end{align}
where $S$ is the set of distinct generating vertices of $\boldsymbol {\omega}$. 

Now from Step 4 of Lemma \ref{lem:sym}, observe that for each set of linear combination, whenever $j\neq 1$, $m_j$ can be written as a linear combination of $\{l_i; 1\leq i\leq j-1\}$ and $m_1$. By abuse of notation let $m_1$ and $l_j, 1\leq j\leq b$ denote the indices of the generating vertices. Then, as $n \rightarrow \infty$, the above sum goes to 

\begin{align}\label{limit-sym}
{\int_0^1 \int_0^1 \cdots \int_0^1} \prod_{j=1}^b g_{k_j}\bigg(\frac{1}{2}- |\frac{1}{2}-|x_{m_j}-x_{l_j}||\bigg)  \ dx_S,
\end{align}
where $dx_S= dx_{m_1}dx_{l_1}\cdots dx_{l_b}$ is the $(b+1)$-dimensional Lebesgue measure.

Now suppose $\boldsymbol {\omega} \in E_b(k)$. Then for any $m \in \mathbb{N}$,  
\begin{align}\label{integral-sym}
f_{2m}(x) &:= \int_{0}^1 g_{2m}\Big(\frac{1}{2}- |\frac{1}{2}-|x-y||\Big)  \ dy \nonumber \\
&= \int_0^{\frac{1}{2}} g_{2m}(t) \ dt + \int_0^{\frac{1}{2}} g_{2m}(t)\ dt= 2  \int_0^{\frac{1}{2}} g_{2m}(t) \ dt = C_{2m}.
\end{align}
So  $f_{2m}(x)$ is independent of $x$.
Then, using \eqref{integral-sym}, \eqref{limit-sym} can be written as 
\begin{align}
\int_0^1 \int_0^1 \cdots \int_0^1 \prod_{j=1}^{b-1} g_{k_j}\bigg(\frac{1}{2}- \big|\frac{1}{2}-|x_{m_j}-x_{l_j}|\big|\bigg) C_{s_b}   \ dx_{m_1} dx_{l_1} \cdots dx_{l_{b-1}}.
\end{align}
Proceeding in this manner, for $\boldsymbol {\omega}\in E_b(k)$, \eqref{limit-sym} can be written as
\begin{align*}
\int_0^1 \prod_{j=1}^b C_{k_j} \ dx_0= \prod_{j=1}^b C_{k_j}.
\end{align*}
As there are  $ \prod_{i=1}^b {{k_i-1} \choose \frac{k_i}{2}}$ set of  equations that contribute identically, the total contribution for the word 
$\boldsymbol {\omega}\in E_b(k)$ is 
\begin{align}\label{limit-symfi}
 \prod_{j=1}^b{{k_j-1} \choose \frac{k_j}{2}} C_{k_j}.
\end{align}
Now suppose $\boldsymbol {\omega} \notin E(2k)$. Suppose $\boldsymbol{\omega}$ contains $b_1$ distinct letters that appear even number of times and $b_2$ number of distinct letters that appears odd number of times and $b=b_1+b_2$.
Using a similar argument as in Step 2 of the proof of Theorem \ref{thm:mainrev}, it is easy to see that the contribution of this $\boldsymbol {\omega}$ is 0 (see \eqref{finitesum-nons2k}).

Therefore $\displaystyle\lim_{n \rightarrow \infty}\frac{1}{n}\mathbb{E}[\Tr(Z_n)^{2k+1}]=0$ for any $k\geq 0$.
 
For any partition $\sigma\in E_b(2k)$, let $\{V_1, \ldots V_b\}$ be its blocks. Then from \eqref{momentnoniid} and \eqref{limit-symfi}, we have
\begin{equation}\label{sym1}
\lim_{n \rightarrow \infty}\frac{1}{n}\mathbb{E}[\Tr(Z_n)^{2k}] = \displaystyle \sum_{b=1}^k \sum_{\sigma \in E_b(2k)} \prod_{j=1}^b \frac{1}{2} {{|V_j|} \choose \frac{|V_j|}{2}} C_{|V_j|}
=  \sum_{\sigma \in E(2k)} a_{\sigma} C_{\sigma},
\end{equation}
where $a_{2n}= \frac{1}{2} {{2n} \choose n} $,  and $a_{\sigma}$ and $C_{\sigma}$ are the multiplicative extensions of the sequence $a_{2n}$ and $C_{2n}$ respectively.
Also we have that $\displaystyle\lim_{n \rightarrow \infty}\frac{1}{n}\mathbb{E}[\Tr(Z_n)^{2k+1}]=0$ for any $k\geq 0$. This establishes the first moment condition.
\vskip3pt

Now, we show that $\displaystyle\lim_{n \rightarrow \infty}\frac{1}{n}\mathbb{E}[\Tr(Z_n)^{2k}]=\gamma_{2k}$ determines a unique distribution. As $a_{2k}\leq 2^{2k}$, we have 
\begin{equation*}
\gamma_{2k}= \lim_{n \rightarrow \infty}\frac{1}{n}\mathbb{E}[\Tr(Z_n)^{2k}]
 \leq \displaystyle \sum_{\sigma \in E(2k)} a_{\sigma}M_{\sigma}
  \leq \displaystyle \sum_{\sigma \in \mathcal{P}(2k)} 2^{2k} M_{\sigma}
= 2^{2k} \alpha_{2k}.
\end{equation*}
Since $\{\alpha_{2k}\}$ satisfies Carleman's condition, $\{\gamma_{2k}\}$ also does so. This proves that there is a unique distribution whose even moments are $\{\gamma_{2k}\}_{k \geq 1}$. 

Hence the EESD of $Z_n$ converges weakly to $\mu$. This completes the proof of part (a).\\
 
 \vskip3pt
 
  To prove part (b) of the theorem, observe that by Lemma \ref{lem:metric},
 \begin{align}\label{d2inequality-sym}
 d_2^2(\mathbb{E}\mu_{SC_n},\mathbb{E}\mu_{Z_n})& \leq \frac{1}{n} \mathbb{E}[\Tr(SC_n-Z_n)^2] \nonumber \\
 & = \begin{cases}
  \frac{1}{n}  \sum_{i\neq 0} 2n \mathbb{E}[ x_{i}^2 \boldsymbol {1}_{[|x_{i}| > t_n]}]+ \frac{1}{n} n \mathbb{E}[ x_{0}^2 \boldsymbol {1}_{[|x_{i}| > t_n]}] & \text{ for } n \text{ odd},\\
  \frac{1}{n}  \sum_{i\neq 0, \frac{n}{2}} 2n  \mathbb{E}[ x_{i}^2 \boldsymbol {1}_{[|x_{i}| > t_n]}] + \frac{1}{n} n \mathbb{E}[ x_{0}^2 \boldsymbol {1}_{[|x_{i}| > t_n]}] + \frac{1}{n} n \mathbb{E} [x_{\frac{n}{2}}^2 \boldsymbol {1}_{[|x_{i}| > t_n]}] & \text{ for }n \text{ even}.
  \end{cases}
 \end{align}
 Now if $\{t_n\}$ also satisfies condition \eqref{truncationsym}, then we see that the right side of \eqref{d2inequality-sym} goes to $0$ and hence from (a) we can say that the EESD of $SC_n$ converges to $\mu$. This proves part (b).
%
	
\subsubsection{A description of the LSD of the Symmetric Circulant}\label{description of limit}
\begin{definition}
For probability measures $\mu$ on $\mathbb{R}$ and $\nu$ on $(0,\infty)$, their product convolution, $\mu \otimes \nu$ is given by 
 $$\mu \otimes \nu(B)= \int_0^{\infty} \mu (x^{-1} B) \nu(dx)= (\mu \times \nu)f^{-1}(B) \text{ with } f(x,y)=xy$$ for any Borel set $B$. This gives a probability measure on $\mathbb{R}$. Let $F_{\mu \otimes \nu}(\cdot)$ be the distribution function of this measure. Then 
 \begin{align}\label{product-conv}
 F_{\mu \otimes \nu}(y)= \int_0^{\infty} \mu \big((-\infty,\frac{y}{x}]\big) \nu(dx).
 \end{align}
\end{definition}

\begin{lemma}\label{a_nmoment}
The sequence $\{a_{2n}\}_{n \geq 1}$ defined in Theorem \ref{thm:mainsym} is the moment sequence of a unique probability distribution. That is, there exists a random variable $Z$ such that for every $n \geq 1$, $\mathbb{E}[Z^{2n}]= a_{2n}$ and 
$\mathbb{E}[Z^{2n-1}]= 0$.
\end{lemma}
 \begin{proof}
Let $\nu$ be the \textit{arcsine law} on $(0,1)$ with density $\rho(x)=\frac{2}{\pi} \frac{1}{\sqrt{1-x^2}}$ for $0<x<1$ and $\mu$ be the \textit{Bernoulli distribution} giving probability $\frac{1}{2}$ to $1$ and $-1$ each.
 
Consider the product convolution $\mu \otimes \nu$ where $\mu$ and $\nu$ are defined as above. Let $G(\cdot)$ be the distribution function of the arcsine law $\nu$. Then from  \eqref{product-conv}, we get
 $$ F_{\mu \otimes \nu}(y)=\left\{\begin{array}{ll}
 0 & \mbox{for } y \leq -1,\\
  & \\
 \int_{-y}^1 \frac{1}{2} \frac{2}{\pi} \frac{1}{\sqrt{1-x^2}} dx = \frac{1}{2} (1 - G(-y)) & \mbox{for }-1\leq y \leq 0,\\
  & \\
 \int_0^y \frac{2}{\pi} \frac{1}{\sqrt{1-x^2}} dx & \mbox{for } 0\leq y \leq 1.\\
 \end{array}
 \right.
 $$
 Therefore 
 \[  F_{\mu \otimes \nu}(y)= \begin{cases} 
      0 & \mbox{for } y \leq -1,\\
       \frac{1}{2} (1 - G(-y)) & \mbox{for } -1\leq y \leq 0,\\
      \frac{1}{2} (1 + G(y)) & \mbox{for } 0\leq y\leq 1, \\
      1 & \mbox{for }y \geq 1 .
   \end{cases}
\] 
 Hence the density of the measure $\mu \otimes \nu$ is given by $$f_{\mu \otimes \nu}(x)= \frac12\rho(|x|), \ \ \ x \in [-1,1].$$ 
 Consider the density function  $$f(x)= \frac{1}{4} \rho\Big(\frac{|x|}{2}\Big),\ \ \ \ x \in [-2,2].$$
 Let $X$ be a random variable with density  $f$. Then for $k\geq 1$, $\mathbb{E}[X^{2k-1}]=0 $ and 
\begin{equation}\label{2kmoment}
\mathbb{E}[X^{2k}]=\frac{1}{2\pi} \int_{-2}^2 \frac{x^{2k}}{\sqrt{1-\frac{x^2}{4}}} dt= \frac{2^{2k}}{\pi} \int_0^1 y^{k- \frac{1}{2}} (1-y)^{-\frac{1}{2}} dy=\frac{2^{2k}}{\pi} \frac{\Gamma{(k+ \frac{1}{2})}\Gamma{(\frac{1}{2})}}{\Gamma{(k+1)}}= {{2k} \choose k}.
\end{equation}  
Now let $Y$ be the random variable which takes values $1$ and $0$ with probability $ \frac{1}{2} $ each. Suppose $Y$ is independent of $X$. Then from \eqref{2kmoment}, it is easy to see that if $Z=XY$, then $\mathbb{E}[Z^{2k}]=a_{2k}$. 
\end{proof}
  


\subsection{Proof of Theorem \ref{thm:maintoe}}  
 \begin{proof}
[\textbf{Proof of Theorem \ref{thm:maintoe}}]
We first prove the theorem for the Toeplitz matrix and then for the Hankel matrix.
\vskip3pt
\noindent \textbf{Toeplitz matrix:} First of all note that the entries of the Toeplitz matrix $Z_n$ can be assumed to have mean zero. This reduction follows from Step 1 in the proof of Theorem \ref{thm:mainrev}.

\vskip3pt 
\noindent  Next we prove the first moment condition for this matrix.

From \eqref{momentf1usual} and using the fact that $\mathbb{E}(y_{i})=0$,  we have 	
\begin{align}\label{momentnoniid-toe}
\lim_{n \rightarrow \infty}\frac{1}{n}\mathbb{E}[\Tr(Z_n)^{k}]& =\displaystyle\lim_{n \rightarrow \infty}\frac{1}{n} \sum_{\pi:\ell(\pi)=k}\mathbb{E}[Y_{\pi}] 
 = \displaystyle\lim_{n \rightarrow \infty} \displaystyle \sum_{b=1}^k  \sum_{\underset{\text{with b distinct letters}}{\omega \ \text{matched}}}\frac{1}{n} \sum_{\pi \in \Pi(\boldsymbol{\omega})} \mathbb{E}(Y_{\pi}).  
\end{align}
Now suppose $\boldsymbol {\omega}$ is a word with $b$ distinct letters each letter appearing $k
_1,k_2,\ldots,k_b$ times. Let the $j$-th distinct letter appears at $(\pi(i_j-1),\pi(i_j))$-th position for the first time. Denote $(\pi(i_j-1),\pi(i_j))$ as $(m_j,l_j)$. 
Let $v_i=\pi({i})/n$ as defined in Lemma \ref{lem:toe} and $U_n= \{0,1/n,2/n,\ldots,(n-1)/n\}$. 

Let $\boldsymbol {\omega}\in E_b(k)$. 
Clearly, as observed in  Lemma \ref{lem:toe}, there are $ \prod_{i=1}^b {{k_i-1} \choose \frac{k_i}{2}}$ combination of equations of the $s_j$'s (and hence $v_j$'s) for determining the non-generating vertices, once the generating vertices are chosen. Let us denote a generic combination of the $v_j$'s by $L_{\boldsymbol{\omega}}^T$ (see \eqref{integral-gen}). For each of the combination of equations we get positive (possibly different) contribution (see Lemma \ref{lem:toe}). 
 Then the contribution of each  combination $L_{\boldsymbol{\omega}}^T$ corresponding to the word $\boldsymbol {\omega}$ in \eqref{momentnoniid-toe} can be written as 
\begin{align}\label{finitesum-toep}
 \frac{1}{n^{b+1}} \sum_S \prod_{j=1}^b g_{k_j,n}\big(|v_{m_j}-v_{l_j}|\big)\textbf{1}(0 \leq v_0+ L_{i}^T(v_S) \leq 1,\ \forall \  i \in S^{\prime}),  
\end{align}
where $S$ is the set of  distinct generating vertices and $S^{\prime}$ is   the set of indices of the non-generating vertices of $\boldsymbol {\omega}$.
 By abuse of notation let $m_1$ and $l_j, 1\leq j\leq b$ denote the indices of the generating vertices. Therefore as $n \rightarrow \infty$, the contribution of $\boldsymbol {\omega}$ in \eqref{momentnoniid-toe} is given by 
\begin{align}\label{limit-toe}
 \sum_{L_{\boldsymbol{\omega}}^T}\int_{0}^{1} \int_{0}^{1}\int_{0}^{1} \cdots  \int_{0}^{1} \prod_{j=1}^b g_{k_j}\big(|x_{m_j}-x_{l_j}|\big)  \textbf{1}(0 \leq x_0+  L_{i}^T(x_S) \leq 1,\ \forall \  i \in S^{\prime})\ dx_S, 
\end{align}
where $dx_S= dx_{m_1}dx_{l_1}\cdots dx_{l_b}$  denotes the $(b+1)$-dimensional Lebesgue measure on $[0,1]^{b+1}$. As $\boldsymbol {\omega}$ is an even word, from Lemma \ref{lem:toe}, it follows that for every set of linear combination $L_{\boldsymbol{\omega}}^T$, there is a set of positive Lebesgue measure in $[0,1]^{b+1}$ where the indicator function in the above integral takes the value 1. 

If $\boldsymbol{\omega}$ is not an even word, then following a  similar argument as given in Step 2 of the proof of Theorem \ref{thm:mainrev}, it can be shown that its contribution is zero in the limit. Therefore for $k\geq 0$, $$\displaystyle\lim_{n \rightarrow \infty}\frac{1}{n}\mathbb{E}[\Tr(Z_n)^{2k+1}]=0$$ and
 \begin{align}\label{toe1}
\gamma_{2k}=& \displaystyle\lim_{n \rightarrow \infty}\frac{1}{n}\mathbb{E}[\Tr(Z_n)^{2k}] \nonumber \\
= & \displaystyle \sum_{b=1}^k \sum_{\sigma \in E_b(2k)} \sum_{L_{\sigma}^T}\int_{0}^{1} \int_{0}^{1} \int_{0}^{1} \cdots  \int_{0}^{1} \prod_{j=1}^b g_{k_j}\big(|x_{m_j}-x_{l_j}|\big)  \textbf{1}(0 \leq x_0+ L_{i}^T(x_S) \leq 1,\ \forall \  i \in S^{\prime})\ dx_S. 
\end{align}
This establishes the first moment condition.

\vskip3pt

\noindent Now we show that $\gamma_{2k}= \lim_{n \rightarrow \infty}\frac{1}{n}\mathbb{E}[\Tr(Z_n)^{2k}]$ determines a unique distribution. 

Observe that the limiting moment sequence is dominated by the moment sequence  of the Symmetric Circulant case (see \eqref{sym1}). As the latter satisfies the Carleman's condition, so does $\gamma_{2k}$.

 \noindent Therefore we see that the there exists a measure $ \mu$ with moment sequence $\{\gamma_{2k}\}$ such that $\mathbb{E}\mu_{Z_n}$ converges weakly to $ \mu$. This completes the proof of  part (a). 

Proof of part (b) is trivial, so we skip it. 
This completes the proof for the Toeplitz matrix.

\vskip5pt
\noindent \textbf{Hankel matrix:} 
 First of all note that the entries of the matrix can be assumed to have mean zero due Step 1 in the proof of Theorem \ref{thm:mainrev}.

\vskip3pt 

\noindent Next we prove the first moment condition for this matrix. From \eqref{momentf1usual} and using the fact that $\mathbb{E}(y_{i})=0$,  we have 	
\begin{align}\label{momentnoniid-han}
\lim_{n \rightarrow \infty}\frac{1}{n}\mathbb{E}\big[\Tr(Z_n)^{k}\big]& =\displaystyle\lim_{n \rightarrow \infty}\frac{1}{n} \sum_{\pi:\ell(\pi)=k}\mathbb{E}[Y_{\pi}] 
 = \displaystyle\lim_{n \rightarrow \infty} \displaystyle \sum_{b=1}^k  \sum_{\underset{\text{with b distinct letters}}{\omega \ \text{matched}}}\frac{1}{n} \sum_{\pi \in \Pi(\boldsymbol{\omega})} \mathbb{E}(Y_{\pi}).  
\end{align}
Now suppose $\boldsymbol {\omega}$ is a symmetric word with $b$ distinct letters. Let $v_i=\pi(i)/n$ and $U_n=\{0,1/n,2/n,\ldots,(n-1)/n\}$ as defined in Lemma \ref{lem:han}. 
  We know that the $v_i$'s satisfy a linear relation given in \eqref{linear-han}. 
 Then the contribution of $\boldsymbol {\omega}$ in \eqref{momentnoniid-han} can be written as 
\begin{align}\label{finitesum-toe}
 \frac{1}{n^{b+1}} \sum_S \prod_{j=1}^b g_{k_j,n}\big(v_{m_j}+v_{l_j}\big)\textbf{1}(0 \leq L_{i}^H(v_S) \leq 1,\ \forall \  i \in S^{\prime}),  
\end{align}
where $S$ is the set of  distinct generating vertices and $S^{\prime}$ is the set of indices of the non-generating vertices of $\boldsymbol {\omega}$.

By abuse of notation let $m_1$ and $l_j, 1\leq j\leq b$ denote the indices of the generating vertices. Therefore as $n \rightarrow \infty$, for each symmetric word $\boldsymbol{\omega}$, the contribution to the limit in  \eqref{momentnoniid-han} is given by (see Lemma \ref{lem:han})
 \begin{align}\label{limit-han}
  {\int_0^1\int_0^1 \cdots \int_0^1} \prod_{j=1}^b g_{k_j}(x_{m_j}+x_{l_j}) \textbf{1}(0 \leq L_{i}^H(x_S) \leq 1,\ \forall \  i \in S^{\prime}) \ dx_S,
 \end{align}
 where $dx_S= dx_{m_1}dx_{l_1}\cdots dx_{l_b}$ is the $(b+1)$-dimensional Lebesgue measure. As $\boldsymbol {\omega}$ is a symmetric word, from Lemma \ref{lem:han}, it follows that there is a set of positive Lebesgue measure in $[0,1]^{b+1}$ where the indicator function in the above integral takes the value 1. 

If $\boldsymbol{\omega}$ is not an symmetric word, then following a very similar argument as given in the proof of Theorem \ref{thm:mainrev} it can be shown that its contribution is zero in the limit. Therefore for $k\geq 0$, $\displaystyle\lim_{n \rightarrow \infty}\frac{1}{n}\mathbb{E}[\Tr(Z_n)^{2k+1}]=0$ and
 \begin{align}\label{han}
\gamma_{2k}&=\displaystyle\lim_{n \rightarrow \infty}\frac{1}{n}\mathbb{E}[\Tr(Z_n)^{2k}] \nonumber\\
&= \displaystyle \sum_{b=1}^k \sum_{\sigma \in S_b(2k)}  {\int_0^1 \int_0^1 \cdots \int_0^1} \prod_{j=1}^b g_{k_j}(x_{m_j}+x_{l_j})\textbf{1}(0 \leq L_{i}^H(x_S) \leq 1,\ \forall \  i \in S^{\prime})  \ dx_S.
\end{align}

This establishes the first moment condition in this case.

\vskip3pt

\noindent Further, \ $\gamma_{2k}= \lim_{n \rightarrow \infty}\frac{1}{n}\mathbb{E}[\Tr(Z_n)^{2k}]$ determines a unique distribution. This is because $\gamma_{2k}$ is dominated by the moment sequence in the Reverse Circulant case (see \eqref{rev-noniid}). As the latter satisfies the Carleman's condition, so does $\gamma_{2k}$. 

 Hence there exists a measure $ \mu$ with moment sequence $\{\gamma_{2k}\}$ such that $\mathbb{E}\mu_{Z_n}$ converges weakly to $\mu$. This completes the proof of part (a).
 
\noindent  Proof of part $(b)$ is trivial, so we skip it.
\end{proof}
 
\section{Discussion}\label{example and discussion} 

In Section \ref{iid} we deduce Results \ref{result:rev}{\textemdash}\ref{result:toe}
from Theorems \ref{thm:mainrev}{\textemdash}\ref{thm:maintoe}.
Then we derive some results on the LSD of sparse matrices and deduce some results from \citep{banerjee2017patterned} in an unified approach. We also give some examples of LSD for matrices with a variance profile in Section \ref{variance profile}. We conclude the discussions by deriving some new results about band matrices in Section \ref{band matrices} and obtaining some of the conclusions of  \citep{basak2011limiting} and \citep{liu2011limit} as special cases of our results.

We first recall a general lemma regrading the proof of ESD using moment method that can be applied to some of the following cases that we will discuss. See Section 1.2 of \citep{bose2018patterned} for its proof.
\begin{lemma}\label{lem:genmoment}
Suppose $A_n$ is any sequence of symmetric random matrices such that the following conditions hold:
\begin{enumerate}
\item[(i)] For every $k\geq 1$, $\frac{1}{n}\mathbb{E}[\Tr (A_n)^k] \rightarrow \alpha_k$ as $n \rightarrow \infty$.
\item[(ii)] $\displaystyle \sum_{n=1}^{\infty}\frac{1}{n^4}\mathbb{E}[\Tr(A_n^k) \ - \ \mathbb{E}(\Tr(A_n^k))]^4  < \infty$ for every $k \geq 1$.
\item[(iii)] The sequence $\{\alpha_k\}$ is the moment sequence of a unique probability measure $\mu$.
\end{enumerate}
Then $\mu_{A_n}$ converges to $\mu$ weakly almost surely.
\end{lemma}

 To verify condition (ii) for the patterned matrices, we shall need an upper bound of the following set.   
\begin{align*}
Q_{k,4}^b = | \{(\pi_1,\pi_2,\pi_3,\pi_4):\ell(\pi_i)=k;\ \pi_i, 1 \leq i \leq 4 \ \text{jointly- and cross-matched with } b \text{ distinct letters} \}|.
\end{align*}
It can be shown that
\begin{equation}\label{four circuits}
|Q_{k,4}^b|\leq n^{2k+2} \ \ \ \text{for any } 1\leq b \leq 2k.
\end{equation}
For a proof of this fact, see Lemma 1.4.3 (a) in \citep{bose2018patterned}. The proof of the lemma given in \citep{bose2018patterned} is for the case where the entries are i.i.d. However the same arguments can be adapted if the entries are independent. 

\subsection{Full i.i.d. entries}\label{iid}
Let $A_n$ be any one of the four $n \times n$ patterned matrices introduced in Section \ref{introduction}. Let the input sequence of $A_n$ be $\{\frac{1}{\sqrt n}{x_{i}}: i \geq 0\}$, 
where $x_{i}$ are independent and identically distributed with mean 0 and variance 1.	

Let $t_n=n^{-1/3}$. Then $ t_n\sqrt{n} \rightarrow \infty$ as $n \rightarrow \infty$ and
\begin{align*}
\displaystyle \lim_{n \rightarrow \infty} n  \mathbb{E}\bigg [\bigg(x_{i}/\sqrt{n}\bigg)^2 \boldsymbol{1}_{[|x_{i}/\sqrt{n}| \leq t_n]}\bigg]  = \displaystyle \lim_{n \rightarrow \infty}   \mathbb{E}\bigg [\big(x_{0}\big)^2 \boldsymbol{1}_{[|x_{0}| \leq \sqrt{n} t_n]}\bigg]=1.
\end{align*}
Also, for any $k>2$,
\begin{align*}
n \mathbb{E}\bigg [\bigg(\frac{x_{i}}{\sqrt{n}}\bigg)^{k} \boldsymbol{1}_{[|\frac{ x_{i}}{\sqrt{n}}| \leq t_n]}\bigg] \ &
= n\ \mathbb{E}\bigg [\bigg(\frac{x_i}{\sqrt{n}}\bigg)^{k-2}  \bigg(\frac{x_i}{\sqrt{n}}\bigg)^{2} \boldsymbol{1}_{[|x_{i}| \leq t_n  \sqrt{n}}\bigg]\\
& \leq  t_n^{k-2}\ \mathbb{E} \big [{x_{i}}^{2} \boldsymbol{1}_{[|y_{i}| \leq {t_n  \sqrt{n}}]}\big]\\
& \leq t_n^{k-2}\\
& =  (n^{- \frac{1}{3}})^{k-2} \ \rightarrow \ 0 \ \ \text{ as } n \rightarrow \infty.
\end{align*}
Now for any $t > 0$, 
\begin{align*}
\displaystyle \sum_{i=0}^{n-1}\mathbb{E}\big[ \big(x_{i}/\sqrt{n}\big)^2 \boldsymbol {1}_{[|x_{i}/\sqrt{n}| > t_n]}\big] & =  \frac{1}{n}\displaystyle \sum_{i=0}^{n-1}\mathbb{E}[ x_{i}^2 \boldsymbol {1}_{[|x_{i}| > {t_n \sqrt{n}}]}]\\
& \leq \frac{1}{n}\displaystyle \sum_{i=0}^{n-1} \mathbb{E}[x_{i}^2\boldsymbol{1}_{[|x_{i}| > t]}] ]\ \ \text{for all large} \ n,\\
& \overset{a.s.}{\longrightarrow}  \mathbb{E}\big[ x_{0}^2[\boldsymbol {1}_{[|x_{0}| > t]}]\big], \ \ \text{as}\ \ n\to \infty.
\end{align*}
As $\mathbb E(x_0^2)=1$, the last term in the above expression tends to zero as $t\to\infty$.

So $\{\frac{1}{\sqrt n}x_i;i\geq 0\}$ satisfy Assumption A  with $g_2\equiv 1$ and $g_{2k}\equiv 0$ for all $k \geq 2$. 
Therefore from Theorems \ref{thm:mainrev}{\textemdash}\ref{thm:maintoe},
the EESD of $A_n$ converges weakly to a symmetric probability measure $\mu$, say. We now identify $\mu$ in each case:
\vskip3pt

\noindent (i) Suppose $A_n$ is $RC_n$.
Then by Theorem \ref{thm:mainrev}, 
\begin{align*}
\beta_{2k}(\mu)= \displaystyle \sum_{\sigma \in S(2k)} C_{\sigma}= \displaystyle \sum_{\sigma \in S_k(2k)} 1= k!.
\end{align*}
The second last equality holds since $C_{2k}=0$ for all $k \geq 2$ as $g_{2k}\equiv 0$ for all $k \geq 2$ and $C_2=g_2=1$. 
Hence $\mu$ is  the \textit{symmetrised Rayleigh distribution}. 
\vskip3pt

\noindent (ii) Now suppose $A_n$ is $SC_n$.
Then by Theorem \ref{thm:mainsym}, 
\begin{align*}
\beta_{2k}(\mu)= \displaystyle \sum_{\sigma \in E(2k)} a_{\sigma}C_{\sigma}= \displaystyle \sum_{\sigma \in E_k(2k)} 1= \frac{(2k)!}{2^k k!}. 
\end{align*}
The second last equality in the above expression holds  since $C_{2k}=0$ for all $k \geq 2$ as $g_{2k}\equiv 0$ for all $k \geq 2$, $C_2=g_2=1$ and $a_2=1$. 
Thus $\mu$ is  the standard normal distribution. 
 \vskip3pt

\noindent (iii) Similarly choosing $A_n$ as the Toeplitz and Hankel matrices we get the convergence of the EESDs.

 Now, as $A_n$ satisfies \eqref{four circuits}, we have 
\begin{align}\label{fourthmoment-noniid}
& \frac{1}{n^{4}} \mathbb{E}\big[\Tr(A_n^k)  -  \mathbb{E}(\Tr(A_n^k))\big]^4 = \mathcal{O}(n^{-2})\ \ \text{ and therefore,}\nonumber\\
&  \displaystyle \sum_{n=1}^{\infty}\frac{1}{n^{4}} \mathbb{E}\big[\Tr(A_n^k)  -  \mathbb{E}(\Tr(A_n^k))\big]^4 < \infty \ \ \ \text{ for every } k\geq 1.
\end{align}
Then using Lemma \ref{lem:genmoment}, we can conclude almost sure convergence of the ESD of $A_n$.\\

\begin{remark} 
It is not too hard to show the following results. We omit the proofs: \vskip3pt

\noindent (i) Suppose the LSD of $RC_n$ in Theorem \ref{thm:mainrev} is the symmetrised Rayleigh distribution.
Then $g_{2k}=0$ for all $k \geq 2$ and $C_2=\int_0^1 g_2(t) \ dt =1$. 
\vskip3pt
 
\noindent (ii) Suppose the LSD of $SC_n$ in Theorem \ref{thm:mainsym}  is the standard normal distribution.
Then  $C_2=\int_0^1 g_2(t) \ dt =1$ and $g_{2k}=0$ for $k\geq 2$. 
\end{remark}

\subsubsection{General triangular i.i.d. entries}\label{triangulariid} 
Let $A_n$ be one of the four $n \times n$ patterned matrices mentioned in Section \ref{introduction}. Suppose for each fixed $n$ the input sequence $\{x_{i,n}: i \geq 0\}$   
are independent and identically distributed  with all moments finite.
 Assume that for all $k \geq 1$,
 \begin{equation}\label{ck}
  n \mathbb{E}[x_{0,n}^k]\rightarrow C_k \ \ \text{ as }\ n \rightarrow \infty.
 \end{equation}
Also assume that the moments of the random variable whose cumulants are $\{0,C_2,0,C_4,\ldots\}$ satisfy Carleman's condition.

Now observe that Assumption A (i), (ii) and (iii)  are satisfied with $t_n=\infty$ and $g_{2k}\equiv C_{2k}$ for $k \geq 1$.
Hence from Theorems \ref{thm:mainrev}{\textemdash}\ref{thm:maintoe}
the EESD of $RC_n$, $SC_n$, $T_n$ and $H_n$ converge weakly to  symmetric probability measures, whose moments are as in the respective theorems.
\begin{remark} Let $\{x_{i,n}: \ i\geq 1\}$ be i.i.d. with all moments finite, for every fixed $n$. Assume that $\sum_{i=1}^{n}x_{i,n}$ converges in distribution to a limit distribution $F$ whose cumulants are $\{C_k\}_{k \geq 1}$.
This assumption is equivalent to condition \eqref{ck}. 

In particular if $F$ is an infinitely divisible distribution with all moments finite, we can definitely find such i.i.d. random variables $\{x_{i,n}: i \geq 1\}$. See page 766 (characterization 1) in \citep{bose2002contemporary}.
\end{remark}
\subsubsection{Sparse triangular i.i.d. entries} 
Suppose the input sequence  $\{x_{i,n}: \ i \geq 0\}$ of the patterned matrices are $\mbox{Ber}(p_n)$ where $np_n \rightarrow \lambda>0$. 
\citep{banerjee2017patterned} showed that the EESD of patterned matrices with such entries converge to symmetric probability measures. Here we give an alternative proof of this fact and also identify which partitions contribute to the limiting moments.

Observe that in this situation, \eqref{ck} is satisfied with $C_k=\lambda$ for all $k\geq 1$.
Therefore from the discussion  in Section \ref{triangulariid} we obtain that the EESD of $A_n$ converges weakly. 

 Let us now look at the particular cases.
\vskip3pt
\noindent (i) If $A_n$ is the sparse reverse circulant matrix, then 
$$\beta_{2k}(\mu)=\displaystyle \sum_{\sigma \in S(2k)} \lambda^{|\sigma|}.$$
Therefore the  \textit{half cumulants}  (as described in Section 3 of \citep{bose2011half}) of $\mu$ are $\{0, \lambda,0,\lambda,\ldots\}$. For more details on half cumulants, see  \cite{bose2011half}. 

\vskip3pt

\noindent (ii) If $A_n$ is the sparse symmetric circulant, then 
$$\beta_{2k}(\mu)=\displaystyle \sum_{\sigma \in E(2k)} a_{\sigma}\lambda^{|\sigma|}.$$
As described in Section \ref{description of limit}, all odd cumulants of $\mu$ vanish and its even cumulants are $\{a_{2n}C_{2n}\}_{n \geq 1}$, where $a_{2k}$ is the $2k$-th moment of a random variable $Z$ defined  in Lemma \ref{a_nmoment}. Now as $C_{2n}=\lambda$ for all $n$, we get that $\mu$ is the symmetrised \textit{compound Poisson distribution} with rate $\lambda$ and jump distribution $Z$, where $Z$ is the distribution identified in Section \ref{description of limit}.

  \vskip3pt
 
\noindent (iii) Now suppose $A_n$ is the sparse Toeplitz matrix. Then its EESD converges weakly to $\mu$ and $\beta_{2k}(\mu)=\displaystyle \sum_{\sigma \in E(2k)} \alpha({\sigma})\lambda^{|\sigma|}$ where $\alpha(\sigma)$ is obtained from the different linear combinations of $s_j$'s corresponding to $\sigma$ (see \eqref{limit-toe}).
  
  \vskip3pt
	
\noindent (iv) Finally, suppose $A_n$ is the sparse Hankel matrix. Then its EESD converges weakly to $\mu$ and $\beta_{2k}(\mu)=\displaystyle \sum_{\sigma \in S(2k)} \alpha({\sigma})\lambda^{|\sigma|}$ where $\alpha(\sigma)$ is obtained as the value of an integral corresponding to $\sigma$ (see \eqref{han4} and \eqref{limit-han}).

 \vskip3pt
 
\begin{remark} 
(\textit{Sums of sparse matrices}) From this discussion on sparse matrices we can also conclude that the EESD of finite sums of sparse matrices with independent entries converge weakly to a symmetric probability measure. This can be observed in the following manner.

Suppose $A_{n,1},A_{n,2},\ldots,A_{n,m}$ are $m$ independent $n\times n$ matrices whose entries are independent $\mbox{Ber}(p_n)$ where $np_n \rightarrow \lambda>0$. Then the entries  $\{x_{i,n}:  i\geq 0\}$ of $A_n:=\sum_{k=1}^m A_{n,k}$  are independent $Bin(m,p_n)$.
Then for $i \geq 0$, 
$$\mathbb{E}[x_{i,n}^k]= m p_n(1-p_n)^{m-1}+ \displaystyle \sum_{j=2}^m j^k {k \choose j}p_n^j(1-p_n)^{m-j}= mp_n+ o(p_n).$$ 
Clearly \eqref{ck} is satisfied with $C_k=m \lambda$ for all $k \geq 1$. Hence, from the above discussion, the EESD of $A_n$ converges to a symmetric probability measure in all of the four cases. In each case the $2k$-th moments of the limiting distribution are obtained accordingly replacing $\lambda$ by $m\lambda$ here. 
\end{remark}
 
Below we give a few simulations that show intuitively how when the entries are iid $N(0,1)/\sqrt{n}$ (top row of Figure \ref{fig:revfig2} and \ref{fig:symfig2}), we have almost sure convergence, however the same is not true when the entries are i.i.d. $Ber(3/n)$ (bottom row of Figure \ref{fig:revfig2} and \ref{fig:symfig2}). 
 \begin{figure}[htp]
\begin{subfigure}{.5\textwidth}
  \centering
  \includegraphics[width=.7\linewidth]{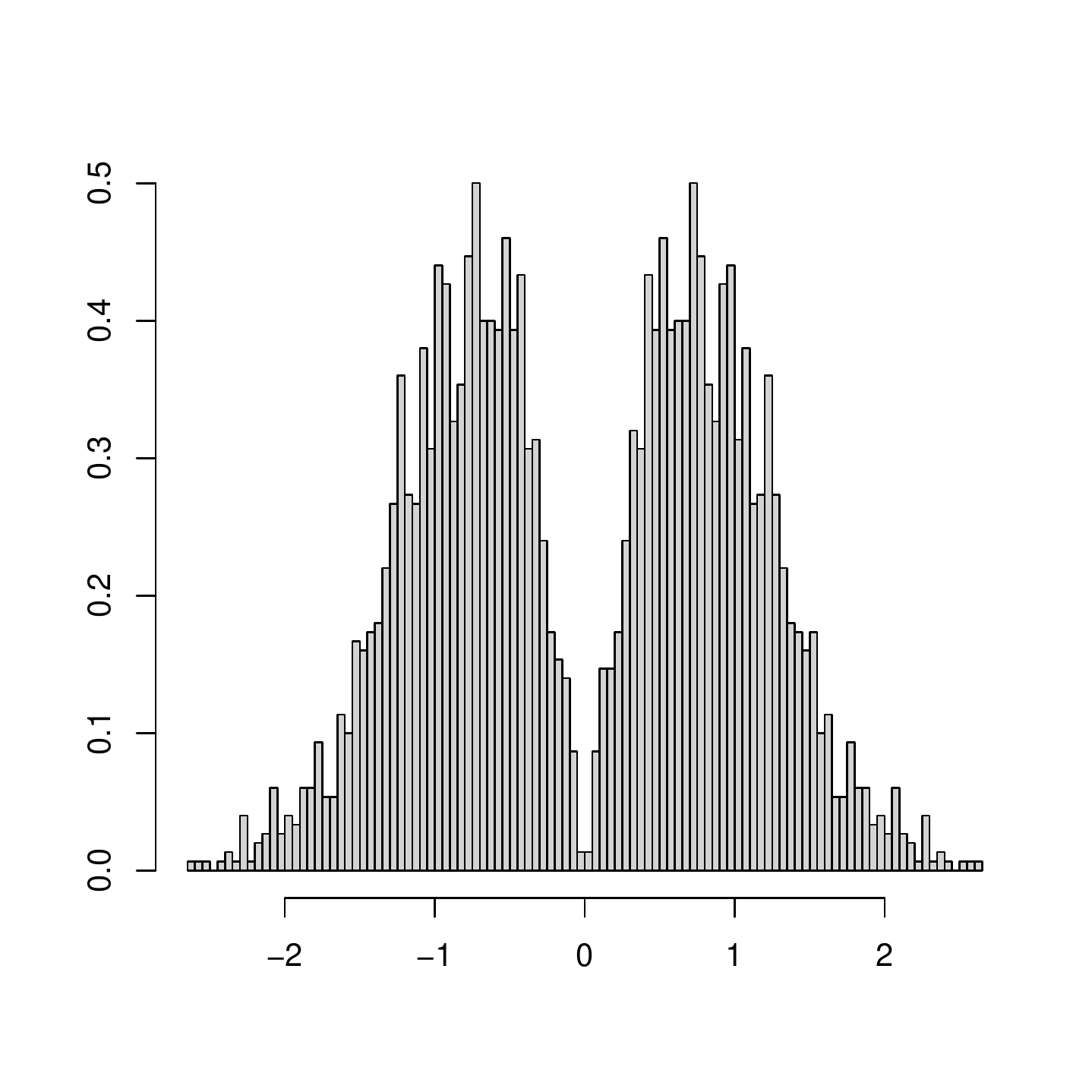}  
\end{subfigure}
\begin{subfigure}{.5\textwidth}
  \centering
  \includegraphics[width=.7\linewidth]{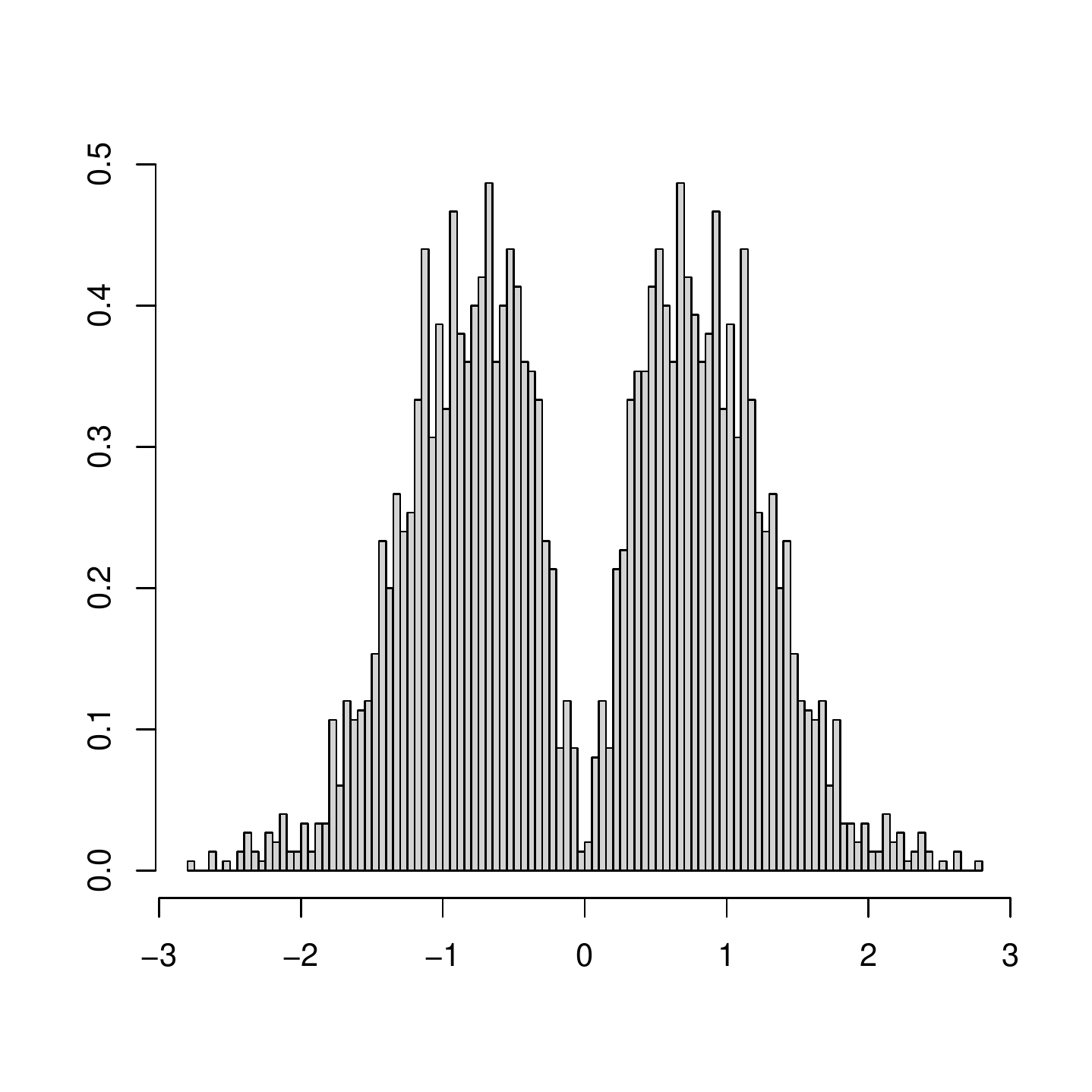}  
\end{subfigure}
\\
\begin{subfigure}{.5\textwidth}
  \centering
  \includegraphics[width=.7\linewidth]{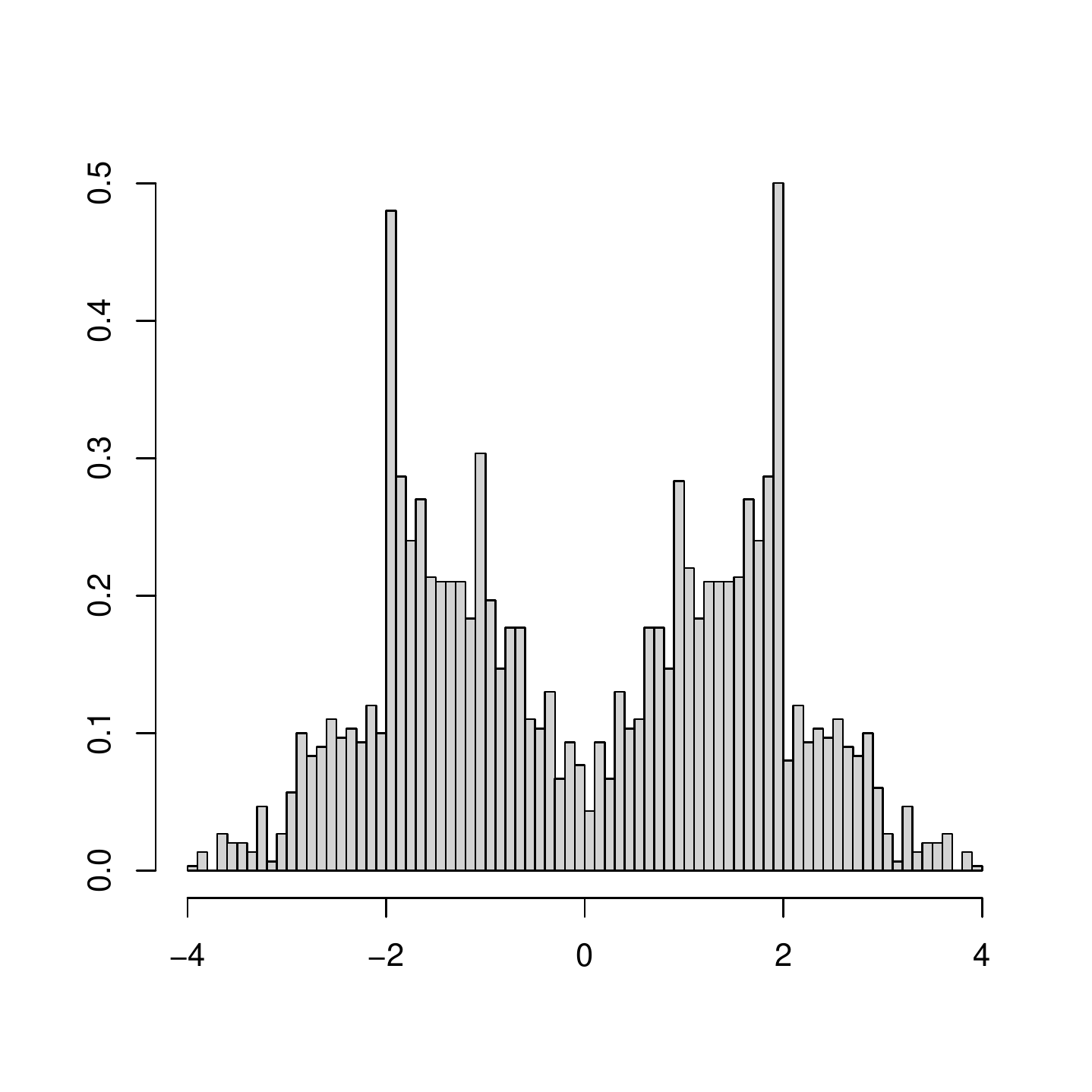}  
\end{subfigure}
\begin{subfigure}{.5\textwidth}
 \centering
  \includegraphics[width=.7\linewidth]{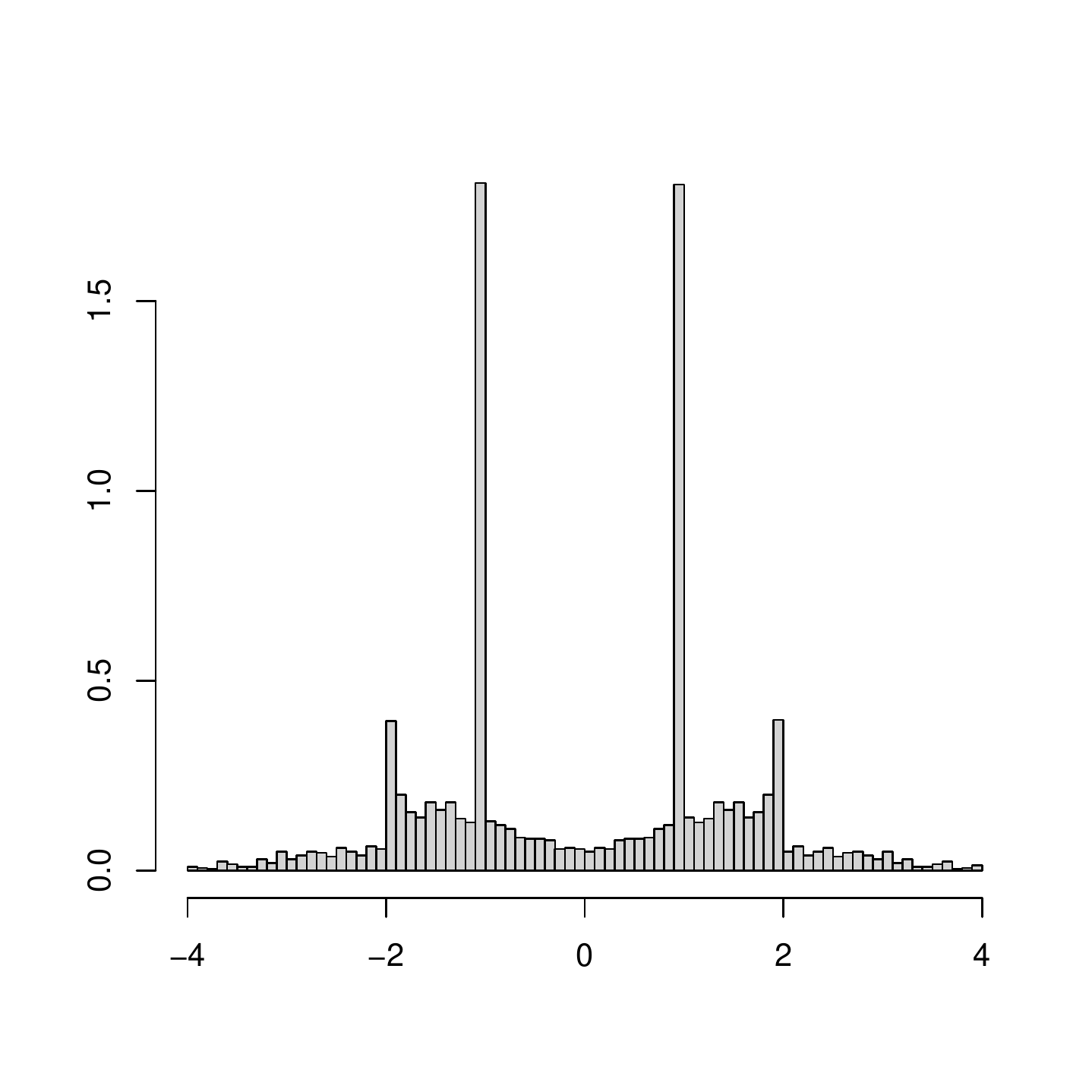}  
\end{subfigure}
\caption{Histogram of the eigenvalues of $RC_n$ for $n=1000$ where the entries are i.i.d. $N(0,1)/\sqrt{n}$ (top row) and  i.i.d. Ber$(3/n)$ (bottom row) for every $n$.}
\label{fig:revfig2}
\end{figure}

 \begin{figure}[htp]
\begin{subfigure}{.5\textwidth}
  \centering
  \includegraphics[width=.7\linewidth]{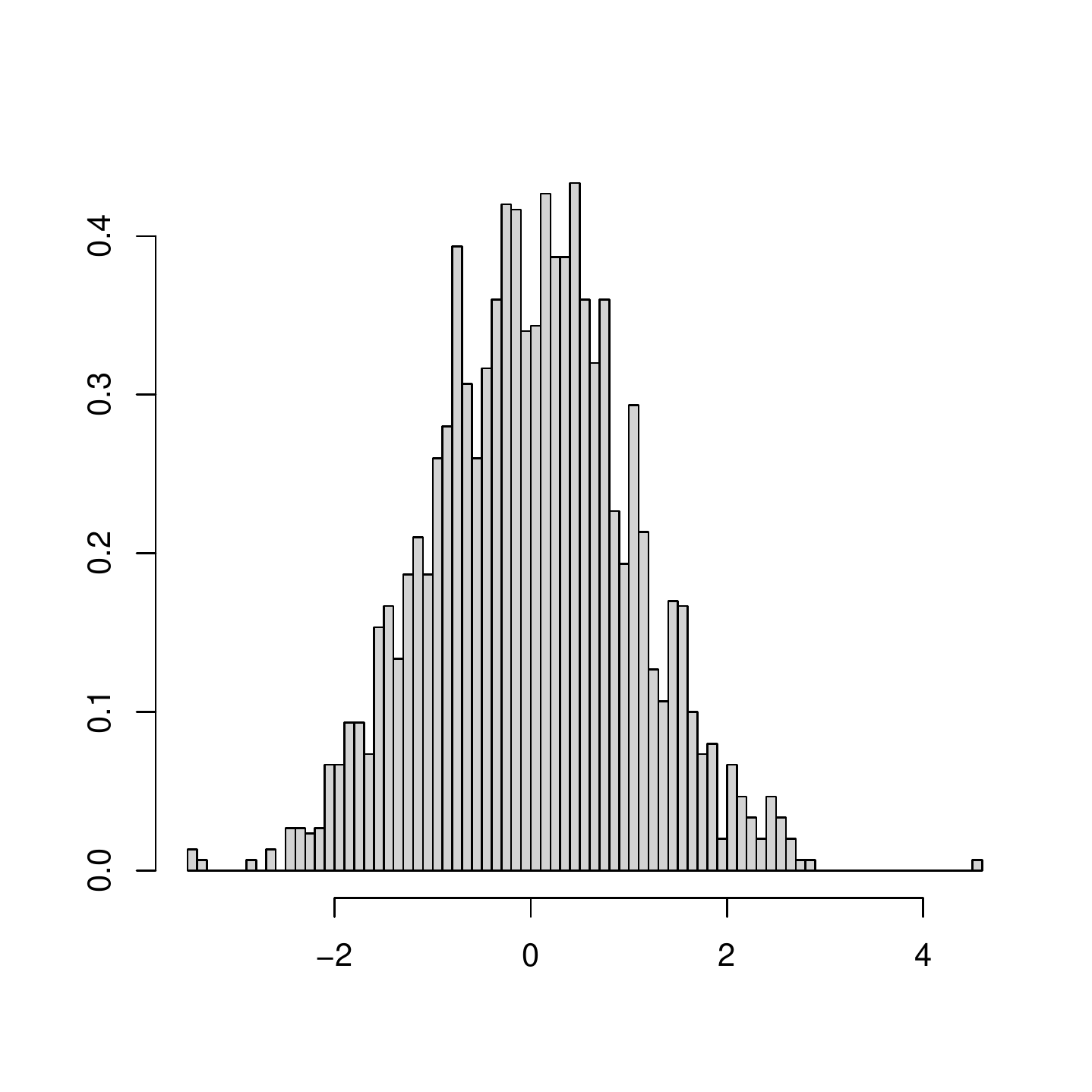}  
\end{subfigure}
\begin{subfigure}{.5\textwidth}
  \centering
  \includegraphics[width=.7\linewidth]{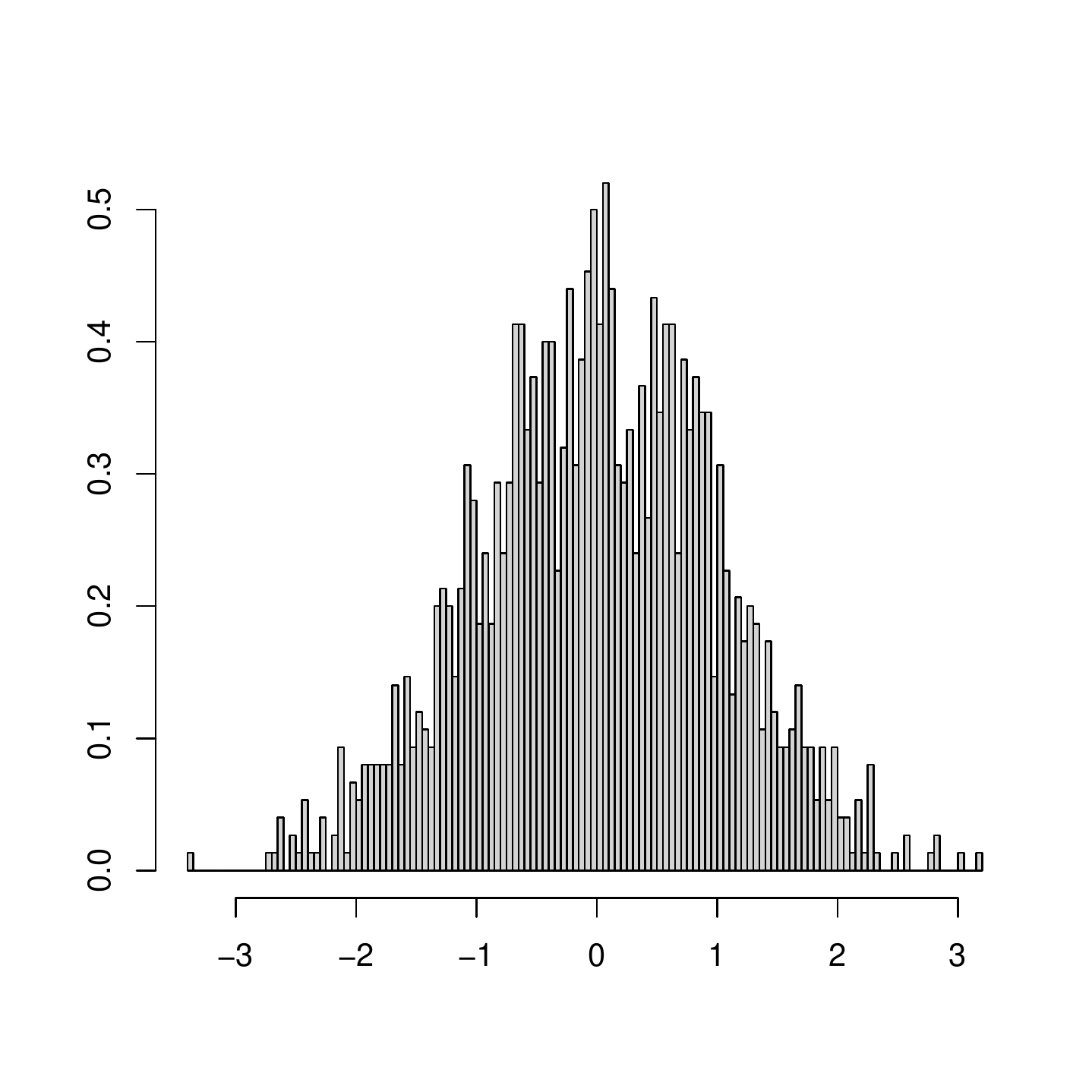}  
\end{subfigure}
\\
\begin{subfigure}{.5\textwidth}
  \centering
  \includegraphics[width=.7\linewidth]{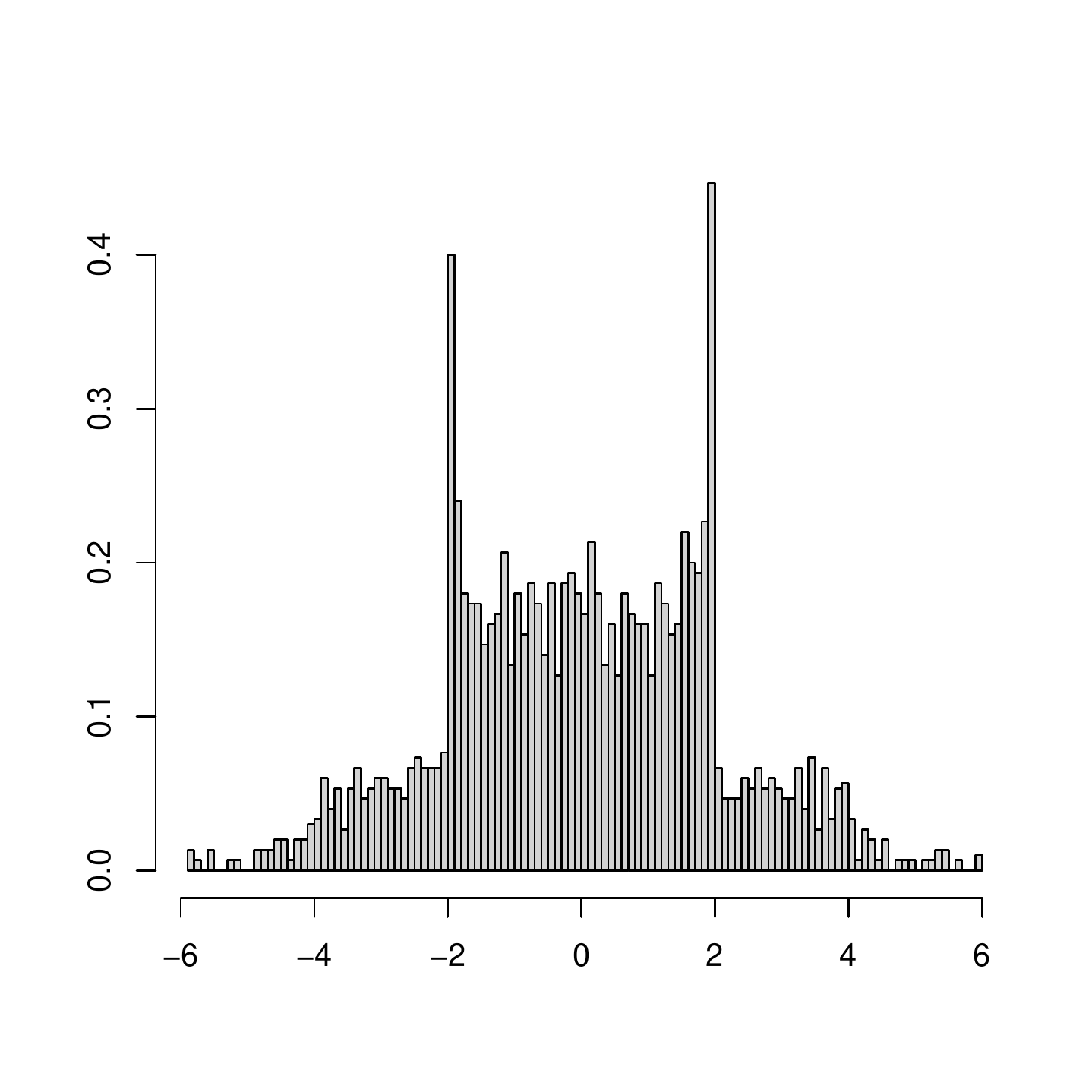}  
\end{subfigure}
\begin{subfigure}{.5\textwidth}
 \centering
  \includegraphics[width=.7\linewidth]{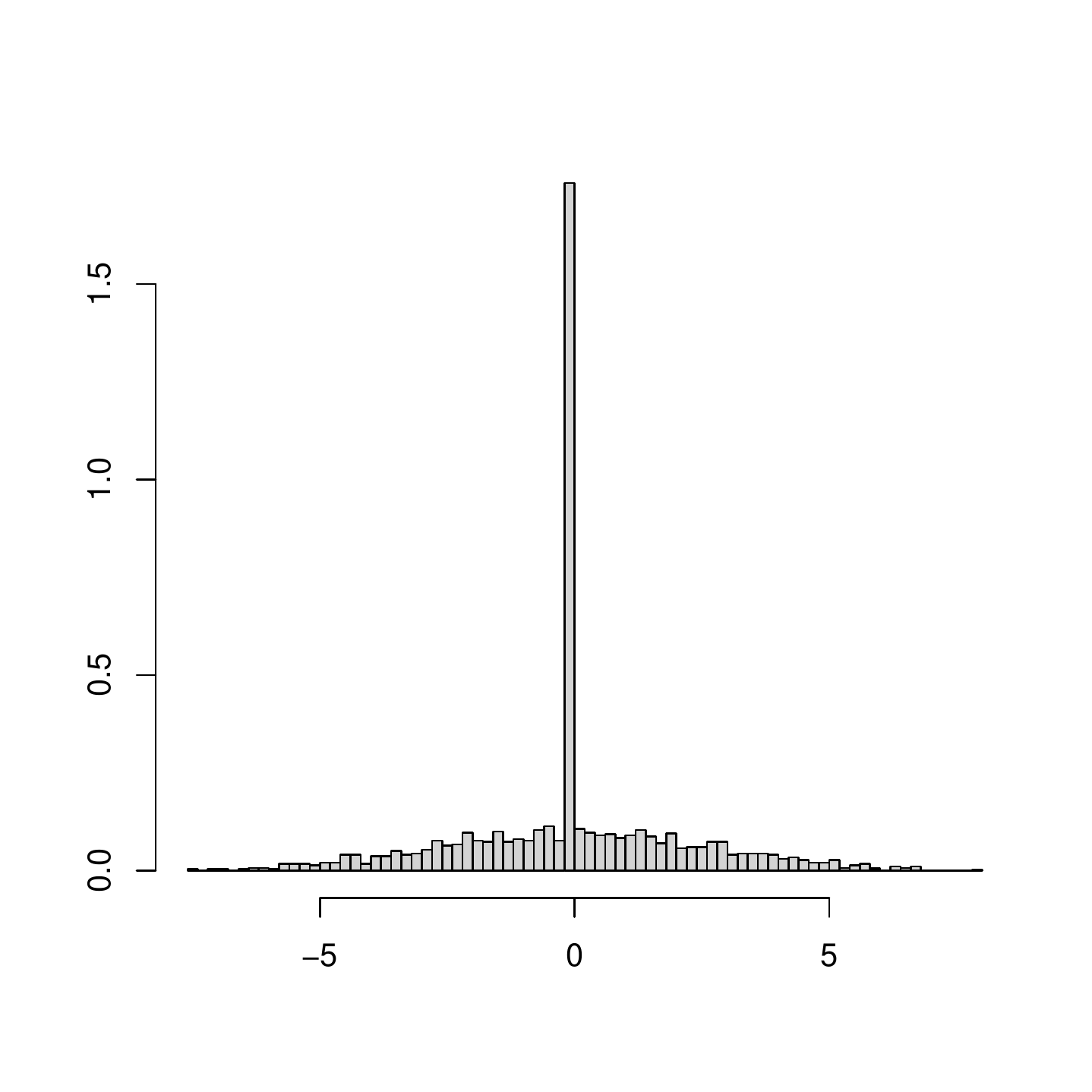}  
\end{subfigure}
\caption{Histogram of the eigenvalues of $SC_n$ for $n=1000$ where the entries are i.i.d. $N(0,1)/\sqrt{n}$ (top row) and  i.i.d. Ber$(3/n)$ (bottom row) for every $n$.}
\label{fig:symfig2}
\end{figure}

\subsection{Matrix with a variance profile}\label{variance profile}
We shall discuss LSD of patterned matrices with two kinds of variance profile: \textit{Discrete variance profile} and \textit{Continuous variance profile}. 

\subsubsection{Discrete variance profile}
Suppose the input sequence of the matrix $A_n$ is  
$\{\sigma_ix_{i,n};i\geq 0\}$, where $\{\sigma_i\}$ is a sequence of real numbers and  $\{x_{i,n};i\geq 0\}$ are  i.i.d. random variables with mean zero and all moments finite. Random matrices with such type of inputs are known as matrices with a \textit{discrete variance profile} or \textit{separable variance profile}.

Assume that $\{x_{i,n};i\geq 0\}$ satisfy \eqref{ck} and the moments of the random variable whose cumulants are $\{0,C_2,0,C_4,0,C_6,\ldots\}$ satisfy  Carleman's condition.

 Further let $\{\sigma_i:i\geq 0\}$  satisfy the following:
	\begin{enumerate}
	\item[(i)] $\displaystyle \sup_i |\sigma_i|\leq c < \infty$.
	\item[(ii)] For any $k\geq 1$, 
	\begin{equation}\label{sigma-limit}
	\frac{1}{n}\displaystyle \sum_{i=0}^{n-1}\sigma_i^{2k} \rightarrow \alpha_{2k}, \mbox{ say}.
	\end{equation} 
\end{enumerate}
Under these conditions, the EESD of $RC_n$ and of $SC_n$ converge weakly to some symmetric probability distribution $\mu_1$ and $\mu_2$, respectively, whose moments are determined by $\{\alpha_{2n}\}_{n\geq 1}$ and $\{C_{2n}\}_{n \geq 1}$. This is justified by the following arguments.

 
\vskip3pt 
 
\noindent \textbf{Reverse Circulant}: Let $C_{2k,n}=n \mathbb{E}[x_{0,n}^{2k}]$.
Observe that  the entries of  $A_n=RC_n$ satisfies \eqref{gkeven} with  $t_n=\infty$ and $g_{2k,n}(\frac{i}{n})=\sigma_i^{2k}C_{2k,n}$ for $0\leq i \leq n-1$. Also because of \eqref{ck}, we have \eqref{gkodd} with $t_n=\infty$.

Now from Step 2 of the proof of Theorem \ref{thm:mainrev}, observe that the contribution of the non-symmetric words is 0. 
Suppose $\boldsymbol {\omega}$ is a symmetric word of length $2k$ with $b$ distinct letters which appear $k_1,k_2, \ldots,k_b$ times. Then using  \eqref{finitesum-rev}, the contribution of this word to the limiting moment is the following: 
\begin{align}
\displaystyle \lim_{n \rightarrow \infty}\frac{1}{n^{b+1}} \displaystyle \sum_{S} \prod_{j=1}^b \sigma_{L(m_j,l_j)}^{k_j}\prod_{j=1}^b C_{k_j,n},
\end{align}
where $L$ is the link function of the reverse circulant matrix, $S$ is the set of generating vertices and $(m_j,l_j)$, $ {1 \leq j \leq b}$ are as in the proof of Theorem \ref{thm:mainrev}.  

Observe that $C_{k_j,n}$ does not depend on the values of the generating vertices  as $\{x_{i,n}\}$ are i.i.d. Also from \eqref{ck} we have $C_{k_j,n}{\rightarrow} C_{k_j}$ as $n\to\infty$.
 
Therefore to obtain the contribution of $\boldsymbol {\omega}$ to the limiting moment it is enough to compute $\displaystyle \lim_{n \rightarrow \infty}\frac{1}{n^{b+1}} \displaystyle \sum_{S} \prod_{j=1}^b \sigma_{L(m_j,l_j)}^{k_j}$. 

Now for any $j \in \{1,2,\ldots,b\}$, $m_j=\pi(i_j-1)$ is either a generating vertex or can be written as a linear combination of the \textit{previous} $l_q$'s. Therefore in any case, $m_j=L_j(\{l_q: 0 \leq q \leq j-1\})$, where $L_j$ denotes the linear representation. As a result, $m_j$ does not depend on the value of $l_j$ and hence does not change with $l_j$. With this observation we obtain that for any $s\geq 1$,
\begin{align*}
\frac{1}{n}\displaystyle \sum_{l_j=1}^n\sigma_{m_j,l_j}^{2s} & = \frac{1}{n} \displaystyle \underset{l_j:\ m_j+l_j-2<n}{\sum} \sigma_{m_j+l_j-2}^{2s}+ \frac{1}{n} \displaystyle \underset{l_j:\ m_j+l_j-2\geq n}{\sum} \sigma_{(m_j+l_j-2)-n}^{2s}\\
&=\frac{1}{n} \displaystyle \sum_{t=m_j-2}^{n-1} \sigma_{t}^{2s} +\frac{1}{n} \displaystyle \sum_{t=0}^{m_j-3} \sigma_{t}^{2s} \\
& =\frac{1}{n} \displaystyle \sum_{t=0}^{n-1} \sigma_{t}^{2s}
 \rightarrow \alpha_{2s} \ \mbox{ as }\ n\to\infty.
\end{align*}
As a consequence, we get
$$\lim_{n \rightarrow \infty}\frac{1}{n^{b+1}} \displaystyle \sum_{S} \prod_{j=1}^b \sigma_{L(m_j,l_j)}^{k_j}=\prod_{j=1}^b\alpha_{k_j}.$$
Hence the contribution of any symmetric word of length $2k$, with $b$ distinct letters that appears $k_1,k_2,\ldots,k_b$ times, to the limiting moment is $\prod_{j=1}^b \alpha_{k_j}C_{k_j}$. 
Hence we have  
\begin{equation}\label{rev-moment}
\displaystyle \lim_{n \rightarrow \infty} \mathbb{E}[\Tr(RC_n)^k]= \begin{cases}
\displaystyle \sum_{\pi \in S(k)} \alpha_{\pi} C_{\pi} & \text{ if } k \text{ is even},\\
\ 0 & \text{ if } k \text{ is odd}.
\end{cases}
\end{equation}
Thus we have verified the first moment condition for the reverse circulant matrix with a variance profile.

As $\displaystyle \sup_i |\sigma_i|\leq c$, 
$$ \displaystyle \sum_{\pi \in E(2k)}c^{|\pi|}C_\pi\leq \max\{c,1\}^k \displaystyle \sum_{\pi \in E(2k)}C_\pi.$$
Since the right side of the above inequality satisfies the Carleman's condition, there exists a probability measure $\mu_1$ such that its moments $\{\beta_k(\mu_1)\}_{k \geq 1}$ are as in \eqref{rev-moment}. Hence the EESD of $RC_n$ converges weakly to $\mu_1$.
\vskip3pt

\noindent \textbf{Symmetric Circulant}: Let $C_{2k,n}=n \mathbb{E}[x_{0,n}^{2k}]$.
Observe that  the entries of $A_n=SC_n$ satisfies \eqref{gkeven} and \eqref{gkodd} with  $t_n=\infty$ and $g_{2k,n}(\frac{i}{n})=\sigma_i^{2k}C_{2k,n}$ for $0\leq i \leq n-1$. 

From the proof of Theorem \ref{thm:mainsym} observe that the words, that are not even, contribute $0$. Now, for any even word $\boldsymbol {\omega}$ of length $2k$ with $b$ distinct letters which appear $k_1,k_2, \ldots,k_b$ times, from \eqref{limit-symfi}, observe that there are $ \prod_{i=1}^b {{k_i-1} \choose \frac{k_i}{2}}$ set of equations that contribute identically. Thus, from \eqref{finitesum-sym} the contribution for each such set of equations is now as follows: 
\begin{align}
\displaystyle \lim_{n \rightarrow \infty}\frac{1}{n^{b+1}} \displaystyle \sum_{S} \prod_{j=1}^b \sigma_{L(m_j,l_j)}^{k_j}C_{k_j,n}
\end{align}
where $L$ is the symmetric circulant link function, $S$ is the set of generating vertices  and $(m_j,l_j)$, ${1 \leq j \leq b}$ are as in  the proof of Theorem \ref{thm:mainsym}.

As in the reverse circulant case, it is enough to compute  $\displaystyle \lim_{n \rightarrow \infty}\frac{1}{n^{b+1}} \displaystyle \sum_{S} \prod_{j=1}^b \sigma_{L(m_j,l_j)}^{k_j}$ in order to obtain the contribution of the word $\boldsymbol{\omega}$ to the limiting moment. 
Note that $m_j=L_j(\{i_q: 0 \leq q \leq j-1\})$. As a result, it does not depend on $l_j$. 
Hence for any $s\geq 1$,
\begin{align}\label{sigmasum-sym}
\frac{1}{n}\displaystyle \sum_{l_j}\sigma_{L(m_j,l_j)}^{2s} = \frac{1}{n} \displaystyle \sum_{t=0}^{m_j-1} \sigma_{t}^{2s} +\frac{1}{n} \displaystyle \sum_{t=m_j}^{\frac{n}{2}-1} \sigma_{t}^{2s} +\frac{2}{n} \displaystyle \sum_{t=0}^{\frac{n}{2}-1} \sigma_{t}^{2s} =2\frac{2}{n} \displaystyle \sum_{t=0}^{n/2-1 }\sigma_{t}^{2s} \rightarrow 2\alpha_{2s}.
\end{align}
As a consequence, we have 
$$\displaystyle \lim_{n \rightarrow \infty}\frac{1}{n^{b+1}} \displaystyle \sum_{S} \prod_{j=1}^b \sigma_{L(m_j,l_j)}^{k_j}= \displaystyle \prod_{j=1}^b 2^{k_j}\alpha_{k_j}.$$ 
 Therefore, the contribution of any even word of length $2k$ with $b$ distinct letters to the limiting moment is $\displaystyle \prod_{j=1}^b 2^{k_j}a_{k_j}\alpha_{k_j}C_{k_j}$, where $a_{2n}= \frac{\1}{2}{{2n} \choose n}$. 
 
Let $\delta_{2k}=2\alpha_{2k}$. Then we have that 
\begin{equation}\label{sym-moment}
\displaystyle \lim_{n \rightarrow \infty} \mathbb{E}[\Tr(SC_n)^k]= \begin{cases}
\displaystyle \sum_{\pi \in E(k)} a_{\pi} \delta_{\pi} C_{\pi} & \text{ if } k \text{ is even},\\
\ 0 & \text{ if } k \text{ is odd}.
\end{cases}
\end{equation}
As before, there exists a probability measure $\mu_2$ such that its moments $\{\beta_k(\mu_2)\}_{k \geq 1}$ are as in \eqref{sym-moment}.
\vskip3pt

Note that in the i.i.d. situation where  each $x_{i,n}$ has the same distribution $F$ for all $i$ and $n$, $C_{2k}=0$ for all $k\geq 2$. Hence the EESD of $RC_n$ converges to a symmetric probability distribution $Y$ such that $Y= \sqrt{\alpha_{2}}\mathcal{R}$, where $\mathcal{R}$ is a random variable with  the \textit{symmetrised Rayleigh distribution}. As $\sigma_i$ is bounded uniformly, \eqref{four circuits} and hence \eqref{fourthmoment-noniid} hold true. Thus we can conclude that in this case the ESD of $RC_n$ converges weakly almost surely to the symmetrised Rayleigh distribution.
 
Similarly, the ESD of the symmetric circulant matrix converges weakly almost surely to $\sqrt{2\alpha_2}N$ where $N$ is a standard Gaussian  variable.

\subsubsection{Continuous vriance profile}\label{cont-var}
Suppose the input sequence is  
$\{\sigma(i/n)x_{i,n};i\geq 0\}$, where $\sigma: [0,1]\rightarrow \mathbb{R}$ is a continuous function and   $\{x_{i,n};i\geq 0\}$ are  i.i.d. random variables with mean zero and all moments finite. 
Such matrices are said to have a \textit{continuous variance profile} or \textit{sampled variance profile}.


Assume that $\{x_{i,n};i\geq 0\}$ satisfy \eqref{ck} and the moments of the random variable whose cumulants are $\{0,C_2,0,C_4,0,C_6,\ldots\}$ satisfy  Carleman's condition.

Then the EESD of $A_n$ where $A_n$ has one of the four patterns, converge weakly to symmetric probability distributions whose moments are determined by $\sigma$ and $\{C_{2k}, k \geq 1\}$. This follows from Theorems \ref{thm:mainrev}\textemdash{}\ref{thm:maintoe} as argued below:

Let $C_{2k,n}= n \mathbb{E}[x_{0,n}^{2k}]$. First observe that the entries  of $A_n$ satisfy  Assumption A (i) and (ii)  with $t_n=\infty$, $g_{2k,n}= \sigma^{2k} C_{2k,n}$ and $g_{2k} = \sigma^{2k} C_{2k}$. 
Since  $\sigma$ is a continuous function on a compact set $[0,1]$, it is  bounded. Therefore, Assumption A (iii) is also true. Hence from Theorems \ref{thm:mainrev}\textemdash{} \ref{thm:maintoe}, we can conclude that the EESD of $A_n$ converges weakly. 

Next we give a brief computation of the limiting moments for reverse circulant and symmetric circulant matrices.
\vskip3pt

\noindent \textbf{Reverse Circulant}:  From Theorem \ref{thm:mainrev}, all  odd moments of the LSD are $0$ and the $2k$th moment is given by $\beta_{2k}= \displaystyle \sum_{\pi \in S(2k)} \alpha_{\pi}C_{\pi}$, where $\alpha_{2m}= \int_0^1 \sigma^{2m}(t) \ dt$.

\vskip3pt

\noindent \textbf{Symmetric Circulant}:  
From Theorem \ref{thm:mainsym}, 
all odd moments of the LSD are $0$ and 
the $2k$th moment  
is given by $\beta_{2k}= \displaystyle \sum_{\pi \in S(2k)} a_{\pi}\alpha_{\pi}C_{\pi}$, where $\alpha_{2m}= 2\int_0^1 \sigma^{2m}(t) \ dt$.

Similarly, the moments of the Toeplitz and Hankel matrices can be calculated using \eqref{toe1} and \eqref{han} in Theorem \ref{thm:maintoe} where the function in the integrand is replaced by $\prod_{j=1}^b C_{k_j} \sigma^{k_j}(|x_{m_j}-x_{l_j}|)$ and $\prod_{j=1}^b C_{k_j} \sigma^{k_j}(x_{m_j}+x_{l_j})$ respectively.


\subsection{Band Matrices}\label{band matrices}
Band matrices are usually defined to be matrices whose elements are non-zero only around the diagonal in the form of a band. As the dimension of the matrices increase, so does the number of non-zero elements around the diagonal. Here we shall see what happens to the LSD of the above four matrices with suitable banding.

Let $A_n$ be any one of the four $n \times n$ patterned matrices introduced in Section \ref{introduction}. We use two types of banding. Let $m_n$ be a sequence of integers such that $m_n\rightarrow \infty$ and $\frac{m_n}{n}\rightarrow \alpha>0$. We do not deal with the case $\alpha=0$ in this article since it is not amenable to the type of arguments we have used so far.  
\vskip3pt
\noindent \textbf{Type I banding}: Type I band matrix $A_n^b$, of 
$A_n$ is the matrix 
with entries $y_{i,n}$ where
\begin{align}\label{typeI}
y_{i,n}=\begin{cases}
x_{i,n} & \text{ if } i \leq m_n,\\
0 & \text{ otherwise. }
\end{cases}
\end{align}
\textbf{Type II banding}: The Type II band versions $RC_n^B$ of $RC_n$ and $T_n^B$ of $T_n$ are defined with input sequence $\hat{y}_{i,n}$ where 
\begin{align}\label{typeII-rev}
\hat{y}_{i,n}=\begin{cases}
x_{i,n} &\text{ if } i\leq m_n \text{ or } i\geq n-m_n,\\
0 & \text{ otherwise}. 
\end{cases}
\end{align}
The Type II band versions $H_n^B$ of $H_n$ is defined with input sequence $\tilde{y}_{i,n}$ where 
\begin{align}\label{typeII-han}
\tilde{y}_{i,n}=\begin{cases}
x_{i,n} & \text{ if } n-m_n \leq i \leq n+m_n,\\
0 & \text{ otherwise}. 
\end{cases}
\end{align}
For example suppose $m_n \sim \big[\frac{n}{3}\big]$ where $[\cdot]$ is the greatest integer function. At $n=5$, 
\begin{align*}
RC_5^b= \begin{bmatrix}
x_0 & x_1 & 0 & 0 & 0\\
x_1 & 0 & 0 & 0 & x_0\\
0 & 0 & 0 & x_0 & x_1\\
0 & 0 & x_0 & x_1 & 0\\
0 & x_0 & x_1 & 0 & 0
\end{bmatrix}, \ \ \ SC_5^b=\begin{bmatrix}
x_0 & x_1 & 0 & 0 & x_1\\
x_1 & x_0 & x_1 & 0 & 0\\
0 & x_1 & x_0 & x_1 & 0\\
0 & 0 & x_1 & x_0 & x_1\\
x_1 & 0 & 0 & x_1 & x_0
\end{bmatrix},\ \ \ T_5^b=\begin{bmatrix}
x_0 & x_1 & 0 & 0 & 0\\
x_1 & x_0 & x_1 & 0 & 0\\
0 & x_1 & x_0 & x_1 & 0\\
0 & 0 & x_1 & x_0 & x_1\\
0 & 0 & 0 & x_1 & x_0
\end{bmatrix},
\end{align*}

\begin{align*}
RC_5^B= \begin{bmatrix}
x_0 & x_1 & 0 & 0 & x_4\\
x_1 & 0 & 0 & x_4 & x_0\\
0 & 0 & x_4 & x_0 & x_1\\
0 & x_4 & x_0 & x_1 & 0\\
x_4 & x_0 & x_1 & 0 & 0
\end{bmatrix}, \ \ \ H_5^B=\begin{bmatrix}
0 & 0 & x_4 & x_5 & x_6\\
0 & x_4 & x_5 & x_6 & 0\\
x_4 & x_5 & x_6 & 0 & 0\\
x_5 & x_6 & 0 & 0 & 0\\
x_6 & 0 & 0 & 0 & 0
\end{bmatrix},\ \ \ T_5^B=\begin{bmatrix}
x_0 & x_1 & 0 & 0 & x_4\\
x_1 & x_0 & x_1 & 0 & 0\\
0 & x_1 & x_0 & x_1 & 0\\
0 & 0 & x_1 & x_0 & x_1\\
x_4 & 0 & 0 & x_1 & x_0
\end{bmatrix}.
\end{align*}
Patterned band matrices have been studied in some previous works, for example in \citep{basak2011limiting}, \citep{liu2011limit}, \citep{popescu2009general} and others. \citep{basak2011limiting} considered the LSD of the scaled Reverse Circulant, Symmetric Circulant, Toeplitz and Hankel band matrices where the scaling depends on the number of non-zero entries in the matrices. \citep{liu2011limit} studied the band Toeplitz and Hankel matrices with a particular scaling and proved that the ESD of these matrices converge weakly almost surely to  symmetric probability distributions which depend on the limiting ratio of the number of non-zero entries to the number of zero entries. \citep{popescu2009general} studied the convergence  of the ESD of a scaled tridiagonal matrix models.

We assume that the variables $\{x_{i,n};i\geq 0\}$ associated with the matrices $A_n^b$ or $A_n^B$ (as in \eqref{typeI}, \eqref{typeII-rev} and \eqref{typeII-han}) are i.i.d. random variables with all moments finite.  We also assume that $\{x_{i,n};i\geq 0\}$ satisfy \eqref{ck} and the moments of the random variable whose cumulants are $\{0,C_2,0,C_4,0,C_6,\ldots\}$ satisfy  Carleman's condition.
Then we have the following results.

\begin{result}\label{bandingI}
For all of the four matrices, the EESD of $A_n^b$ converge weakly to some symmetric probability measures $\mu_{\alpha}$ that depend on $\{C_{2k}\}_{k\geq 1}$ and $\displaystyle \alpha=  \lim_{n \rightarrow \infty}\frac{m_n}{n}>0$.  
\end{result}
 
\begin{proof}
For every $n$, 
 define the function $\sigma_n$ on the interval $[0,1]$ as 
$$\sigma_n(x)=\left\{\begin{array}{ll}
1 & \mbox{if }x\leq {m_n}/{n},\\
0 & \mbox{otherwise}.
\end{array}
\right.$$ 
Now observe that the entries $y_{i,n}$ of the matrix $A_n^b$ can be written as $\sigma_n\big({i}/{n}\big)x_{i,n}$. 


Observe that, for any $k \geq 1$, $\int \sigma_n^{k}(x)\  dx \rightarrow \int \sigma^{k}(x)\  dx$ as $n \rightarrow \infty$, where $\sigma$ is defined as 
\begin{align}\label{sigma}
\sigma(x)=\begin{cases}
1 & \text{ if } 0 \leq x \leq \alpha,\\
0 & \text{ otherwise}.
\end{cases}
\end{align}
Following the  proofs in Theorems \ref{thm:mainrev}\textemdash{}\ref{thm:maintoe} and the above convergence of $\int \sigma_n^{k}(x)\  dx$, 
it is easy see that  the first moment condition also hold for  $A_n^b$. 
Hence the EESD of  $A_n^b$ converges weakly to a symmetric probability distribution. 

The formulae for the limiting moments are as follows:
\vskip3pt

\noindent \textbf{Reverse Circulant}:  Clearly from Theorem \ref{thm:mainrev}, the odd moments of the LSD are all zero and  
the $2k$-th moment of the limiting distribution is given by 
$$\beta_{2k}= \displaystyle \sum_{\pi \in S(2k)} \alpha^{|\pi|}C_{\pi}, \mbox{ as } \int_0^1 \sigma^{2m}(t) \ dt= \int_0^{\alpha} \sigma^{2m}(t) \ dt= \alpha.$$

\noindent \textbf{Symmetric Circulant}: Clearly from Theorem \ref{thm:mainsym}, the odd moments of the LSD are all zero and 
the $2k$-th moment of the limiting distribution is given by 
$$\beta_{2k}= \displaystyle \sum_{\pi \in S(2k)} (2\alpha)^{|\pi|}a_{\pi}C_{\pi}, \mbox{ as } 2\int_0^1 \sigma^{2m}(t) \ dt=2\int_0^{\alpha} \sigma^{2m}(t) \ dt= 2 \alpha.$$
Similarly, the moments of the Toeplitz and Hankel matrices can be calculated from \eqref{toe1} and \eqref{han} in Theorem \ref{thm:maintoe} where the limits of the integral is from $0$ to $\alpha$ and the function in the integrand is replaced by $\prod_{j=1}^b C_{k_j} \sigma^{k_j}(|x_{m_j}-x_{l_j}|)$ and $\prod_{j=1}^b C_{k_j} \sigma^{k_j}(x_{m_j}+x_{l_j})$ respectively.
\end{proof}

\begin{remark}
\noindent (i) If the entries of $A_n^b$ are $\frac{y_{i,n}}{\sqrt{m_n}}$ where $\{y_{i,n};i \geq 0\}$ are as in \eqref{typeI} and $\{x_{i,n};i\geq 0\}_{n \geq 1}$ are i.i.d. random variables with mean 0 and variance 1, then $C_2= \frac{1}{\alpha}$ and $C_{2k}= 0$ for  $k\geq 2$. Hence from Result \ref{bandingI}, we obtain the convergence of the EESD. Again it can be verified using the same arguments as in proof of Lemma 1.4.3 in \citep{bose2018patterned} that \eqref{four circuits} and \eqref{fourthmoment-noniid} hold true. Thus we obtain Theorem 1 in \citep{basak2011limiting} for the case $\alpha\neq 0$.\\

\noindent(ii) Suppose $x_{i,n}= \frac{x_{i}}{\sqrt{(2-\alpha)\alpha n}}$ where $\alpha= \displaystyle \lim_{n \rightarrow \infty}\frac{m_n}{n}>0$, $\{x_{i};i\geq 0\}$ are independent variables with mean 0, variance 1 and $\displaystyle \sup_{i} \mathbb{E}[|x_i|^k]=M_k<\infty$ for $k\geq 3$. Now consider the matrix $T_n^b$ with entries $\{y_{i,n}; i \geq 0\}$, where  $\{y_{i,n};i \geq 0\}$ are as  in \eqref{typeI}. Then following the argument in the proof of Theorem \ref{thm:maintoe}, one can see that $g_2\equiv 1$, $g_{2k}\equiv 0$ for $k\geq 2$, and the EESD of $T_n^b$ converges weakly almost surely to a symmetric probability distribution. Similarly as (i), it can be verified using the same arguments as in proof of Lemma 1.4.3 in \citep{bose2018patterned} that \eqref{four circuits} and \eqref{fourthmoment-noniid} hold true. Hence we obtain Theorem 2.2 of \citep{liu2011limit} when $T_n$ is real symmetric.
\end{remark}
\begin{result}\label{bandingII}
The EESD of $A_n^B$ where $A_n$ is either reverse circulant or Toeplitz or Hankel matrix, converges weakly to some symmetric probability measures $\mu_{\alpha}$ that depend on $\{C_{2k}\}_{k\geq 1}$ and $\alpha= \displaystyle \lim_{n \rightarrow \infty}\frac{m_n}{n}>0$.  
\end{result}
 
\begin{proof}
For the Type II band versions $RC_n^B$ of $RC_n$ and $T_n^B$ of $T_n$, for every $n$ we define $\sigma_{n,1}$ on $[0,1)$ as  $$\sigma_{n,1}(x)= \begin{cases}
 1 &\text{ if } x \leq m_n/n \text{ or } x \geq 1-m_n/n,\\
 0 & \text{ otherwise}.
 \end{cases}$$ 
Clearly, $\int \sigma_{n,1}^{k}(x) \ dx \rightarrow \int \sigma_1^{k}(x) \ dx$ as $n \rightarrow \infty$ for each $k \geq 1$ where $\sigma_1= \boldsymbol{1}_{[0,\alpha] \cup [1-\alpha,1]}$.
 
 For the Type II band versions $H_n^B$ of $H_n$, for every $n$, we define a function $\sigma_{n,2}$ on $[0,2)$ as
 $$\sigma_{n,2}(x)= \begin{cases}
 1 &\text{ if } 1-m_n/n\leq x \leq 1 + m_n/n, \\
 0 & \text{ otherwise}.
 \end{cases}$$ 
Clearly, $\int \sigma_{n,2}^{k}(x) \ dx \rightarrow \int \sigma_2^{k}(x) \ dx$ as $n \rightarrow \infty$ for each $k \geq 1$ where  $\sigma_2= \boldsymbol{1}_{[1-\alpha,1+\alpha]}$.

Following the  arguments as in the proof of Result \ref{bandingI}, the convergence of the EESD of  $A_n^B$ follows. 
Now we compute the moments of the LSD of these matrices.
\vskip3pt

\noindent \textbf{Reverse Circulant}:  Clearly 
all odd moments of the LSD are $0$. 
From Theorem \ref{thm:mainrev}, the $2k$th moment of the LSD 
is given by $$\beta_{2k}= \displaystyle \sum_{\pi \in S(2k)} (2\alpha)^{|\pi|}C_{\pi}, 
\mbox{ as } \int_0^{\alpha}\sigma_1^{2m}(t) \ dt + \int_{1-\alpha}^1 \sigma_1^{2m}(t) \ dt=2\alpha.$$
Similarly, the moments of the LSD for Toeplitz and Hankel matrices can be calculated by using (\ref{toe1}) and (\ref{han}) in Theorem \ref{thm:maintoe} where the function in the integrand is replaced by $\prod_{j=1}^b C_{k_j} \sigma_1^{k_j}(|x_{m_j}-x_{l_j}|)$ and $\prod_{j=1}^b C_{k_j} \sigma_2^{k_j}(x_{m_j}+x_{l_j})$ respectively. As $\int_{0}^{1} \sigma_{1}^{2m}(t) \ dt$ and $\int_{0}^{1} \sigma_{2}^{2m}(t) \
dt$ are dependent on $\alpha$ for any $m\geq 1$, the limiting distribution also depends on $\alpha$.
\end{proof}

\begin{remark} 
 (i) If the entries of $A_n^b$ are $\frac{y_{i,n}}{\sqrt{m_n}}$ where $\{y_{i,n};i \geq 0\}$ are as  in \eqref{typeII-rev} or \ref{typeII-han} and $\{x_{i,n};i\geq 0\}_{n \geq 1}$ are i.i.d. random variables with mean 0 and variance 1, then $C_{2k}= 0$ for all $k\geq 2$ and $C_2= \frac{1}{\alpha}$. Hence from Result \ref{bandingII}, we get the convergence of the EESD. Further, it can be verified using the same arguments as in proof of Lemma 1.4.3 in \citep{bose2018patterned} that \eqref{four circuits} and \eqref{fourthmoment-noniid} hold true. Hence we obtain Theorem 1 of  \citep{basak2011limiting} when $\alpha\neq 0$.
\vskip3pt
\noindent(ii) Suppose $x_{i,n}= \frac{x_{i}}{\sqrt{(2-\alpha)\alpha n}}$ where $\alpha= \displaystyle \lim_{n \rightarrow \infty}\frac{m_n}{n}>0$, $\{x_{i};i\geq 0\}$ are independent random variables with mean 0, variance 1 and $ \sup_{i} \mathbb{E}[|x_i|^k]=M_k<\infty$ for $k\geq 3$. Now consider the matrix $H_n^B$ with entries $\{y_{i,n}; i \geq 0\}$, where  $\{y_{i,n};i \geq 0\}$ are as in \eqref{typeII-han} with $x_{i,n}$ as defined above. Then following the argument in Theorem \ref{thm:maintoe}, one can see that $g_2\equiv 1$, $g_{2k}\equiv 0$ for $k\geq 2$ and the EESD of $H_n^B$ converges weakly to a symmetric probability distribution. Further, it can be verified using the same arguments as in proof of Lemma 1.4.3 in \citep{bose2018patterned} that \eqref{four circuits} and \eqref{fourthmoment-noniid} hold true. Hence we obtain the result of Theorem 2.3 of \citep{liu2011limit}.
 \end{remark}

\subsection{Triangular matrices}
Triangular matrices have gained importance since their consideration in \citep{Dykema2002DToperatorsAD}, where the authors considered the triangular Wigner matrices with Gaussian entries. Later in \citep{basu2012spectral}, the authors concluded LSD results about triangular matrices with i.i.d input for some matrices with other patterns such as, Hankel, Toeplitz and symmetric circulant. 

Let $A_n$ be one of the $n \times n$ Hankel, Toeplitz or symmetric circulant matrices. Then triangular $A_n$, denoted by $A_n^u$ is the matrix whose entries $y_{i,n}$ are as follows:
\begin{align}\label{triangular}
y_{i,n}= \begin{cases}
x_{i,n} & \ \ \text{ if } (i+j)\leq n+1\\
0       & \ \ \text{ otherwise}.
\end{cases}
\end{align}
Note that the triangular reverse circulant matrix is the same as triangular Hankel matrix.

\noindent We assume that the variables $\{x_{i,n};i\geq 0\}$ associated with the matrices $A_n^u$ (as in \eqref{triangular}) are i.i.d. random variables with all moments finite, for every fixed $n$.  We also assume that $\{x_{i,n};i\geq 0\}$ satisfy \eqref{ck} and the moments of the random variable whose cumulants are $\{0,C_2,0,C_4,0,C_6,\ldots\}$ satisfy  Carleman's condition.
Then we have the following result.
\begin{result}\label{res:triangular}
For all of the three matrices mentioned above, the EESD of $A_n^u$ converge weakly to some symmetric probability measures $\mu_{A}$ that depend on $\{C_{2k}\}_{k\geq 1}$. 
\end{result}
 
\begin{proof}
The proof of this follows from the same argument given in  Result \ref{bandingI} with $\sigma$ being replaced by $\eta:[0,1]^2\rightarrow \mathbb{R}$ such that $\eta(x,y)=\boldsymbol{1}_{[x+y\leq 1]}$. We skip the details.
\end{proof}

\begin{remark}
\noindent (i) If the entries of $A_n^u$ are $\frac{y_{i,n}}{\sqrt{n}}$ where $\{y_{i,n};i \geq 0\}$ are as in \eqref{triangular} and $\{x_{i,n};i\geq 0\}_{n \geq 1}$ are i.i.d. random variables with mean 0 and variance 1, then $C_2= 1$ and $C_{2k}= 0$ for  $k\geq 2$. Hence from Result \ref{triangular}, we obtain the convergence of the EESD. Again it can be verified that \eqref{four circuits} and \eqref{fourthmoment-noniid} are true in this case. Thus we achieve that the ESD of $A_n^u$ converges weakly almost surely to a non-random symmetric probability measure.
\end{remark}


\providecommand{\bysame}{\leavevmode\hbox to3em{\hrulefill}\thinspace}
\providecommand{\MR}{\relax\ifhmode\unskip\space\fi MR }
\providecommand{\MRhref}[2]{%
  \href{http://www.ams.org/mathscinet-getitem?mr=#1}{#2}
}
\providecommand{\href}[2]{#2}

\bibliographystyle{plainnat}
\bibliography{mybibfilefinal}

\end{document}